\documentclass[a4paper]{amsart}
\usepackage[utf8]{inputenc}
\usepackage[T1]{fontenc}

\usepackage{times}

\usepackage{xcolor}

\usepackage[pagebackref,colorlinks=true,pdfpagemode=none,urlcolor=blue,
linkcolor=blue,citecolor=blue]{hyperref}

\usepackage{amsmath,amsfonts,amssymb,amsthm} 
\usepackage{amsthm, amssymb, amsmath, amsfonts, mathrsfs}

\usepackage{mathrsfs}
\usepackage{MnSymbol}
\usepackage{scalerel} 

\usepackage{color}
\usepackage{soul}
\usepackage{cancel}
\usepackage{accents}

\usepackage{bbm}

\definecolor{labelkey}{gray}{.8}
\definecolor{refkey}{gray}{.8}

\definecolor{darkred}{rgb}{0.9,0.1,0.1}
\definecolor{darkgreen}{rgb}{0,0.5,0}

\newcommand{\tnorm}[1]{{\left\vert\kern-0.25ex\left\vert\kern-0.25ex\left\vert #1 
    \right\vert\kern-0.25ex\right\vert\kern-0.25ex\right\vert}}

\newtheorem{theorem}{Theorem}[section]
\newtheorem{lemma}[theorem]{Lemma}

\newtheorem{corollary}[theorem]{Corollary}
\newtheorem{proposition}[theorem]{Proposition}

\theoremstyle{remark}

\renewenvironment{proof}[1][Proof]{ {\itshape \noindent {#1.}} }{$\Box$
\medskip}

\numberwithin{equation}{section}

\newcommand{\rH}{\mathrm{H}}
\newcommand{\rC}{\mathrm{C}}
\newcommand{\cal}{\mathcal}
\newcommand{\al}{\alpha}
\newcommand{\R}{\mathbb{R}}
\newcommand{\bbR}{\mathbb{R}}
\newcommand{\bbZ}{\mathbb{Z}}

\newcommand{\PP}{\mathbf{P}}
\newcommand{\bbT}{\mathbb{T}}

\newcommand{\Om}{\Omega}
\newcommand{\lam}{\lambda}

\newcommand{\Q}{\mathcal{Q}}

\newcommand{\U}{\mathcal{U}}

\newcommand{\eps}{\varepsilon}

\newcommand{\la}{\langle}
\newcommand{\ra}{\rangle}

\newcommand{\om}{\omega}
\newcommand{\EE}{\mathbf{E}}

\newcommand{\cP}{\mathcal{P}}

\newcommand{\bT}{\mathbb{T}}
\newcommand{\cZ}{\mathcal{Z}}

\definecolor{orange}{rgb}{1.0, 0.5, 0.0}

\newcommand{\mm}[1]{{\color{black}#1}}

\begin{document}

\title{The Gaussian structure of a perturbed KPZ}

\date{\today }

\author{Yu Gu, Tomasz Komorowski}

\address[Yu Gu]{Department of Mathematics, University of Maryland, College Park, MD, 20742 USA}

\address[Tomasz Komorowski]{Institute of Mathematics, Polish Academy of Sciences, ul.
Śniadeckich 8, 00-636 Warsaw, Poland}

\begin{abstract}
 We study the KPZ equation on a circle with an additive spatial
 perturbation $\partial_t h=\tfrac12\Delta h+\tfrac12|\nabla h|^2+
 \xi+ V$, where $\xi$ is a spacetime white noise and $V$ is a smooth
 spatial function. When $V=0$, it is well-known that the unique
 invariant measure is the Brownian bridge. In the presence of the
 perturbation, we show that the equation admits a unique invariant
 measure that is absolutely continuous with respect to the Brownian
 bridge. We further prove the measure has a finite relative entropy
 with respect to the law of the bridge and that, for any $p\in(1,\infty)$, the corresponding  Radon-Nikodym derivative belongs to $L^p$, provided that $\int_{\bT} V^2(x)dx$ is sufficiently small. The proof uses the discretization and mollification scheme of \cite{FQ}, together with an application of the log-Sobolev and spectral gap inequalities for the underlying Gaussian measure.
\end{abstract}
\maketitle

\section{Introduction}\label{intro}

\textbf{Main result. }The goal of this paper is to obtain \emph{quantitative} information on the invariant measures of the KPZ equation under a certain type of perturbations. The KPZ equation takes the form \begin{equation}\label{e.maineq}
\partial_th=\tfrac12\Delta h+\tfrac12|\nabla h|^2+\xi.
\end{equation} It is well-known from  \cite{BG97} that, if $\xi$ is a 1+1 spacetime white noise, the above equation admits the two-sided Brownian motion with an arbitrary drift as an invariant measure, modulo a height shift. The uniqueness of such invariant measures was proved in \cite{DS26,HM18,JRS22}. While the Brownian motion is scaling invariant, one should connect its ``large scale'' diffusive behavior to the roughness exponent $\tfrac12$ in the 1+1 KPZ universality class, and it has been an important open problem to prove such large scale diffusive behaviors for general growth models. 
%

While the most general cases are out of reach, it is tantalizing  to study some perturbed version of \eqref{e.maineq}. For example, one may consider a smooth noise, which seems natural as the small scale properties of the randomness should not affect the large scale statistical behaviors of the solutions. Along this line, there is the work on the case of a  kick forcing \cite{BL19} and the case of a spatial mollification of the white noise \cite{DGR21}. For the constructed measures however, almost nothing is known beyond  some basic regularity and integrability. A different way of perturbing the equation is to keep   the noise singular to stay in the ``elliptic'' setting and ask what happens to the Brownian invariance if we add a smooth perturbation. For example, in the periodic setting, one may smoothly perturb the covariance structure by removing a few low frequency modes of the noise. 
 
  In this paper, we consider a simple perturbation through shifting the noise by  a deterministic, smooth, spatial function $V(\cdot)$, and study the following equation in the periodic setting \begin{equation}\label{e.maineq1}
\begin{aligned}
\partial_th=\tfrac12\Delta h+\tfrac12|\nabla h|^2+\xi+V.
\end{aligned}
\end{equation}
 We assume that $V:\bT\to\R$ is smooth and $\xi$ is a spacetime white noise on $\R\times \bT$, with $\bT$ the unit circle. 
 
Due to the global growth of the height function, the invariant measure for the KPZ refers to that of $(h(t,\cdot)-h(t,0))$, or the spatial derivative $(u(t))=(\nabla h(t))$, which solves the corresponding stochastic Burgers equation:
\begin{equation}\label{e.631}
\partial_tu=\tfrac12\Delta u+\tfrac12\nabla u^2+\nabla \xi+\nabla V.
\end{equation} Here are the main results of the paper (see Theorems~\ref{thm012504-26} and \ref{cor022404-26} for a precise statement):

(i) there exists a unique invariant measure $\nu_V$ for $(u(t))$, and $u(t)$ converges exponentially fast to it as $t\to\infty$;
 
(ii) $\nu_V$ is absolutely continuous with respect to $\nu_0$, the zero-mean spatial white noise;

(iii) $\nu_V$ has a finite relative entropy with respect to
$\nu_0$. For any $p\in(1,\infty)$, if we further assume $\int_{\bT}
V^2(x)dx\leq  Cp^{-1}$, then $\nu_V$ has an $L^p$ density. Here
  $C$ is some universal constant ($1/100$ in our formulation, see
  \eqref{sNaNxz} below).

\medskip

In terms of the height function, the above result shows that $h(t,\cdot)-h(t,0)$ stabilizes exponentially fast in large time, and the steady state $\mathfrak{h}(\cdot)$ has  the same pathwise properties as a standard Brownian bridge on $\bT$. One can also derive some statistical properties, e.g., the Gaussian tail decay of $\PP(|\mathfrak{h}(x)-\mathfrak{h}(0)|>M)$ as $M\to\infty$, provided that $V$ is small as in (iii).

\bigskip
\textbf{Related works.}  In a loose sense one may view
\eqref{e.maineq1} as a continuous counterpart of the slow bond problem
\cite{JJ92,JJ94} (or fast bond depending on the sign of $V$). Put it
in this context, the interesting question would be to understand how
the average growth speed $\lim_{t\to+\infty} h(t,0)/t$ changes
with respect to the perturbation. It is natural to expect the
invariant measure to play an important role, although our result does
not give precise enough information for such an analysis. For the slow bond problem, the question was answered in \cite{BVS14}, and it was further proved in \cite{BSS17} that the steady state is diffusive on large scales.

An interesting example that we are unable to treat   is the singular case $V(\cdot)=\delta(\cdot)$. It is unclear to us in this case if the invariant measure is absolutely continuous with respect to the Gaussian measure, but we believe it is an interesting problem in its own right to try to describe the invariant measure for this special perturbation. It is worth pointing out that this type of singular perturbation arises naturally for the equation with inhomogeneous Neumann boundary conditions. For the half-line and open KPZ equation, the invariant measures are known explicitly and are absolutely continuous with respect to the Brownian motion, see \cite{DGR25} and the references therein. To see the connection between the two problems, one may observe that imposing the inhomogeneous Neumann boundary conditions is equivalent with forcing the equation with Dirac potentials at the boundaries. However, the key difference from our setting is the symmetry embedded into those problems: to reformulate the open KPZ equation as a periodic problem, one needs to impose Dirac potentials at the boundaries which convert the inhomogeneous Neumann boundary conditions into homogeneous ones, then perform an even extension of the equation. The approach developed in \cite{DGR25} to identify the explicit invariant measure for the open KPZ equation relies crucially on the evenness of the noise after this extension, and it does not appear to generalize to \eqref{e.maineq1} with $V=\delta$. That being said, in several places of this paper we treat the shift of the noise, $\xi\mapsto \xi+V$, by means of the Cameron-Martin theorem, which is inspired by \cite{DGR25}.

Another class of related works include \cite{CHT25,MMS22},  where the main focus has been on comparing the law of the nonlinear equation with that of the linear one, based on the time-shifted Girsanov method  developed in \cite{MS05,MS08}. The main issue there is to analyze the relevant \emph{regularity} so that the shift by the nonlinearity lives in the corresponding Cameron-Martin space. For the KPZ equation, this problem does not exist because the invariant measures for the nonlinear and linear equations coincide. One may perturb the noise covariance structure and ask the same question about absolute continuity, and this is not what we do in this paper. The raised question may be approached by the method developed in the aforementioned works, and our approach does not cover it. However, the method here  leads to quantitative estimates on the density which is our main motivation from the beginning. In this regard, we mention the   very interesting recent works \cite{BDW25,DHYZ25,GHR25} in the study of the QFT measure which gives very strong quantitative estimates.  

\bigskip
\textbf{Our approach.} The approach adopted in this paper can be
viewed as an infinite dimensional version of the standard PDE method
in proving the decay of the solution to a heat equation  on a
circle. Since it is a simple argument which may be applied to other
problems, we sketch it informally below. The key is a certain type of anti-symmetry associated with the nonlinearity of the equation.

First, consider the unperturbed case of $V=0$. On a physics level, one can write the generator associated with the stochastic Burgers equation as $\mathcal{L}=\mathcal{S}+\mathcal{A}$, where $\mathcal{S}$ corresponds to the linear part of the equation $\tfrac12\Delta u+\nabla \xi$ and $\mathcal{A}$ corresponds to the nonlinear part $\tfrac12\nabla u^2$. Take a functional $F(u)$ with $u=(u_x)_{x\in\bT}$,  we write on a formal level
\[
\begin{aligned}
&\mathcal{S}F(u)=\int_{\bT} \,  \delta_{u_x} F(u)\tfrac12\Delta u_x dx-\tfrac12\int_{\bT^2} \,\delta^2_{u_xu_y}F(u)\delta''(x-y) dxdy,\\
&\mathcal{A} F(u)=\int_{\bT} \,\delta_{u_x} F(u)\tfrac12\nabla u_x^2 dx ,
\end{aligned}
\]
where $\delta_{u_x}$ is the derivative with respect to $u_x$. An
integration by parts shows that $\mathcal{S}$ is symmetric and
$\mathcal{A}$ is anti-symmetric with respect to the   inner product
with an underlying white noise measure $\nu(du)=e^{-\tfrac12\int_{\bT} \, u_x^2 dx} du$:
\[
\la F, G\ra=\int F(u) G(u)e^{-\tfrac12\int_{\bT} \, u_x^2 dx} du,
\] 
which implies that $\mathcal{L}^*=\mathcal{S}^*+\mathcal{A}^*=\mathcal{S}-\mathcal{A}$.  Suppose $F_t(u)$ is the solution to the Fokker-Planck equation 
\[
\partial_t F_t=\mathcal{L}^* F_t=\mathcal{S} F_t-\mathcal{A} F_t,
\]
started at some density $F_0$. Then it is straightforward to use the
anti-symmetry of $\mathcal{A}$, with respect to $\nu$, to derive 
\begin{equation}\label{e.632}
\begin{aligned}
\tfrac{d}{dt}\la F_t, \log F_t\ra= \la F_t, \mathcal{S}\log F_t\ra, \quad\quad \tfrac{d}{dt} \la F_t,F_t\ra=\la F_t, \mathcal{S}F_t\ra.
\end{aligned}
\end{equation}
In other words, the nonlinear part of the diffusion disappears after an integration by parts on the level of the Fokker-Planck equation. Since $\mathcal{S}$ is the generator of an infinite dimensional OU process, we have the standard log-Sobolev and spectral gap inequalities associated with it, which give the dissipation on the r.h.s.:
\begin{equation}\label{e.633}
\begin{aligned}
\tfrac{d}{dt}\la F_t, \log F_t\ra \leq -c_1\la F_t, \log F_t\ra, \quad\quad \tfrac{d}{dt} \la F_t,F_t\ra \leq -c_2\la F_t-1,  F_t-1\ra,
\end{aligned}
\end{equation}
where $c_1,c_2>0$ are the Log-Sobolev and spectral gap constants. 
This formal argument shows that $F_t\to1$, thereby the law of $u_t$
converges to the white noise. So far this is rather standard and
well-known, on the physics level. A very interesting work \cite{GP20} made it rigorous on the mathematical level. The anti-symmetry of $\mathcal{A}$ may be viewed as the fundamental reason why the nonlinear stochastic Burgers equation admits the same invariant measure as the linear Edwards-Wilkinson equation.

The main contribution of our work is to rigorize and generalize the above argument to the perturbed case of $V\neq0$. With the perturbation, the generator takes the form $\mathcal{L}_V=\mathcal{S}+\mathcal{A}+\mathcal{A}_V$ with $\mathcal{A}_V$ corresponding to the additional term $\nabla V$ in \eqref{e.631}:
\[
\mathcal{A}_V F(u)=\int_{\bT} \, \delta_{u_x}F(u) \nabla V(x) dx .
\]
As a result, we get an extra term in \eqref{e.632} 
\[
\begin{aligned}
&\tfrac{d}{dt}\la F_t, \log F_t\ra= \la F_t, \mathcal{S}\log F_t\ra+\la F_t,\mathcal{A}_V \log F_t\ra, \\
& \tfrac{d}{dt} \la F_t,F_t\ra=\la F_t, \mathcal{S}F_t\ra+\la F_t, \mathcal{A}_V F_t\ra.
\end{aligned}
\]
Applying the same functional inequalities, we derive differential inequalities similar to \eqref{e.633} but with an extra forcing term associated with the perturbation $V$. An application of Gronwall's inequality provides the uniform in time bound on the relative entropy and $L^2$ norm, which leads to our main result.

The above argument is surprisingly simple, at least conceptually. The main technical difficulty lies in the singular nature of the equation, as everything sketched above is purely  formal. On the mathematical side, one may try to follow the approach developed in \cite{GP20} to write a proof on a ``continuous'' level. Alternatively, one may discretize or mollify the equation and proceed by approximation. With an appropriate discretization that preserves the aforementioned (anti-)symmetry, the sketched argument can be put on a firm mathematical ground. This is what we  do in this paper. 

\medskip

Throughout the paper, we assume $\bT$ to be the circle of length
$1$. The same argument applies to arbitrary length $\ell$, which leads
to the absolute continuity of the invariant measure with a finite
relative entropy. However, the obtained estimates deteriorates as
$\ell$ increases, e.g., to obtain $L^p-$integrability of the density,
one would need to assume $\int_{\bT} V^2(x)dx$ is smaller than some
$\ell-$dependent factor and the bound on the $L^p-$norm also depends
on $\ell$ which does not seem realistic. This inefficiency  is not
surprising though, because the log-Sobolev and spectral gap constants
in \eqref{e.633} scales as $\ell^{-2}$ as $\ell\to\infty$, and we used
these functional inequalities for the worst possible
scenarios. Although the anti-symmetric part $\mathcal{A}$ does not
show up in the energy estimate \eqref{e.632}, which apparently
simplifies our analysis, it helps to ``stir'' the equation and move
the mass to high order Fourier modes. As a result, the actual decay
given by the dissipative term $\la F_t, \mathcal{S}\log F_t\ra$, or
$\la F_t, \mathcal{S}F_t\ra$ would be much stronger than the one given
by the functional inequalities. On a philosophical level, this is
similar to those anomalous dissipation or relaxation enhancement
problems, where the incompressible drift induces the anti-symmetric
part of the generator for the underlying diffusion, and although this
anti-symmetric part does not show up in energy estimate, it helps to
improve the dissipation.  It would be very interesting to prove such
phenomenon for our problem as $\ell\to\infty$, in particular, we would
like to obtain $L^p-$norm  bounds on the density that is uniform in $\ell$. Recalling the discussion at the beginning of the paper,  the object of interest, the large scale roughness, can be measured by how $\EE |\mathfrak{h}(x)-\mathfrak{h}(0)|^2$ grows as $x\to\infty$, where $\mathfrak{h}$ is the steady state of the perturbed KPZ. Now, by the absolute continuity of the invariant measure, it is rewritten as 
\[
\EE |\mathfrak{h}(x)-\mathfrak{h}(0)|^2=\EE\big[|B(x)|^2\mathcal{Q}_\ell(B)\big], 
\] where $B$ is a Brownian bridge on the circle of length $\ell$ and $\mathcal{Q}_\ell$ is the density of law of $\mathfrak{h}$ with respect to the Brownian bridge. After applying Cauchy-Schwarz inequality, it is bounded from above by $x\sqrt{\EE\mathcal{Q}_\ell(B)^2}$ for $1\ll x\ll \ell$, where  we borrowed the diffusive behavior from the Brownian bridge and are left with estimating $\mathcal{Q}_\ell(B)$. In other words, the result of this paper provides a new perspective of proving the diffusive upper bound for the steady state, which hinges on estimating $\EE\mathcal{Q}_\ell(B)^2$ as $\ell\to\infty$.

\bigskip




\bigskip
\textbf{Outline of the paper.}  The rest of the paper is organized as follows.  In Section~\ref{prelim}, we introduce the setup and state the main results of the paper. In Section~\ref{ASFM}, we follow our approach developed in \cite{GK} to show the existence, uniqueness and exponential stability of the invariant measure for the perturbed Burgers equation. The proof of the main results, namely, the absolute continuity and the quantitative estimates on the relative entropy and $L^p-$density, will be done in Sections~\ref{Disc}-\ref{s.last}, which is based on a discretization and mollification originally used in \cite{FQ}. Some technical results are left in the appendix.

\section{Preliminaries and the statements of the main results}

\label{prelim}

{In this section, we introduce notations and state our main results.}

\subsection{Stochastic Burgers equation}
\label{sec1.1}

For the equation posed on the torus $\bT=\bbR/\bbZ$,  we assume that
$\Big(\xi(t,x)\Big)$ is the space-time white noise defined over a
probability space $(\Om,{\cal F},\PP)$, i.e.
\begin{equation}
  \label{WN}
  \EE [\xi(t,x)\xi(s,y)]=\delta(t-s) \delta
  (x-y) , \quad (t,x),(s,y)\in\bbR\times\bbT.
\end{equation}
Here $\EE$ is the expectation w.r.t. measure $\PP$.
Then, the process $W(t,x)=\int_0^t\xi(s,x)ds$ is the cylindrical
Wiener process on $L^2(\bT)$. With some abuse of notations, sometimes we shall write
$dW(t,x)=\xi(t,x)dt$.
Denote by $({\cal
  W}_t\big)$ the natural filtration of $\Big(W(t,x)\Big)$.

Throughout the paper we assume that $V:\bT\to\R$ satisfies
\begin{equation}
    \label{V00}
    V\in C^2(\bT)\quad\mbox{and}\quad  \int_{\bT}V(x)dx=0.
  \end{equation}
  By $C^m(\bT)$ (resp. $H^m(\bT)$) $m=0,1,2,\ldots$, we
denote the space of all $m$ times continuously differentiable
functions (resp. the Sobolev space of
functions with square integrable $m$-th derivatives). By $H^{-m}(\bT)$
we denote the dual space to $H^m(\bT)$.  In particular
$H^0(\bT)=L^2(\bT)$ and $C^0(\bT)=C(\bT)$ is the space of
continuous functions.

Consider the stochastic heat equation (SHE) with a perturbed noise 
\begin{equation}\label{e.she}
d Z_V=\Big(\frac12\Delta Z_V+  Z_V  V\Big)dt+ Z_VdW(t,x),  \quad\quad (t,x)\in(0,+\infty)\times\bT.
\end{equation}
If $Z(0,\cdot)\in L^1(\bT)$ and is non-negative,
the mild solution to \eqref{e.she} is defined globally and is
 {positive} a.s., see e.g. \cite{flores,mueller,walsh}. Define 
\[
h_V(t)= \log Z_V(t), \quad\quad u_V(t)=\nabla h_V(t),
\]
so, formally, $h_V, u_V$ solve the perturbed version of the KPZ equation
\[
  \partial_t h_V=\frac12\Delta h_V+\frac12 |\nabla h_V|^2 +\xi+V\]
and the stochastic Burgers equation (SBE)
\begin{equation}
  \label{011804-26}
  \begin{aligned}
d u_V(t,x)=\Big(\frac12\Delta u_V(t,x)+\frac12 \nabla u_V^2(t,x)+\nabla V(t,x)\Big)dt +d\nabla W(t,x).
\end{aligned}
\end{equation}
We shall omit the subscript $V$ in the case $V=0$.

\subsection{Some basic notation}

\subsubsection{Probability measures on a Polish metric space} In what follows we denote by 
$\|\cdot\|_\infty$   the supremum metric on the respective
function space. In addition, given a Polish metric space $(X,{\rm d})$, we let
  ${\cal M}_1(X)$ be the space
 of all Borel probability measures on $ X$, equipped with the
 topology of weak convergence of measures. Denote also by $B_b(X)$
 (resp. $C(X)$, $C_b(X)$) the space of all bounded, Borel
 measurable (resp. continuous, bounded continuous) functions on  $X$.

%


 \subsubsection{Some function spaces}

Denote by $\rC\subset C(\bT)$ the space  of all positive continuous
functions functions defined on $\bT$, 
equipped with the metric 
\begin{equation}
  \label{rC}
  {\rm d}_{\rC}(Z_1,Z_2):=\|\log Z_1-\log Z_2\|_\infty,\quad Z_1,Z_2\in \rC.
\end{equation}
Note that $(\rC, {\rm d}_{\rC})$ is a Polish metric space.  

Consider ${\rm H}$  the space consisting of all elements $u\in
H ^{-1 }(\bT)$ for which there exists a {continuous $h$}
such that 
\begin{equation}
  \label{012704-26}
  \begin{split}
  h(0)=0\quad \mbox{ and}\quad
 u(\varphi)=\mm{-}\la h,\nabla \varphi\ra,\quad \varphi\in
H^1(\bbT) .
\end{split}
\end{equation}
Throughout the paper  $\la\cdot,\cdot\ra=\la \cdot,
\cdot\ra_{L^2(\bT)}$ denotes the usual scalar product in
the $L^2(\bT)$.
One can easily see $h={\frak p}(u)$ is determined  uniquely  {by the relation 
\[
h(x)=\int_0^x u(y)dy, \quad\quad u\in {\rm H}.
\]}  
On $\rH$ we consider the metric
\begin{equation}
  \label{rH}
  {\rm d}_{\rH}(u_1,u_2):=\|{\frak p}(u_2)-{\frak
    p}(u_1)\|_\infty,\quad u_1,u_2\in \rH.
  \end{equation}
$(\rH,  {\rm d}_{\rH})$ is continuously embedded in
$H^{-1 }(\bT)$ and is Polish.

Let 
$D_+(\bT)$ be the space of  all strictly positive, continuous densities 
on $\bT$. 
We equip $D_+(\bT)$ with the topology induced by the metric
\begin{equation}
  \label{D}
  {\rm d}_{D}(\rho_1,\rho_2):=\|\log\rho_1-\log\rho_2\|_\infty,\quad
  \rho_1,\rho_2\in D_+(\bT).
\end{equation}
One can  check that  $\big(D_+(\bT),  {\rm d}_{D}\big)$ is a
Polish metric space.


Furthermore, we let
$R: D_+(\bT)\to \rH$, where
\begin{equation}
  \label{012804-26}
  R(\rho)(\varphi):=-\la\log \rho,\nabla\varphi\ra,\quad \varphi\in
  H^1(\bT),\, \rho\in D_+(\bT).
  \end{equation}
The mapping  is a homeomorphism between $(D_+(\bT),{\rm d}_{D})$ and
$(\rH,{\rm d}_{\rH})$. The inverse $R^{-1}: \rH\to  D_+(\bT)$ is   given
by
\begin{equation}
  \label{012804-26-1}
   R^{-1}(u)(x)=\frac{e^{ h(x)}}{\int_{\bT}e^{h(x')}dx'},\quad\mbox{where} \quad 
  h={\frak p}(u) ,\,  x\in
  \bT.
\end{equation}

Denote by $ D(\bT)$ the set of all continuous densities on $\bT$ equipped with the
metric
$$
{\rm d}_\infty(\rho_1,\rho_2):=\|\rho_2-\rho_1\|_\infty.
$$ 
Note that $(D_+(\bT),{\rm d}_{D})$ is continuously embedded into
$(D(\bT),{\rm d}_{\infty})$ and  $D_+(\bT)$ is a dense subset of
$D(\bT)$ in the supremum metric.

\subsubsection{Mean-zero white noise on   $\bT$}

\label{sec-gwn}

Let $\nu\in{\cal M}_1(\rH)$ be the spatial mean-zero white noise on
$\bT$,
i.e. a 
  Gaussian  $\rH$-valued process such that
  \begin{align*}
    \int_{\rH}  u(\varphi_1)   \nu(du)&=0\quad\mbox{and}\\
    \int_{\rH}  u(\varphi_1)  u(\varphi_2)\nu(du)&=\la
      \varphi_1,\varphi_2\ra_{L^2(\bT)}\\
      &-\Big(\int_{\bT}\varphi_1(x)dx\Big) \Big(\int_{\bT}\varphi_2(x)dx\Big),\quad  \varphi_1,\varphi_2\in
      C^\infty(\bT).
  \end{align*}
The measure  coincides with  the law of 
\begin{align}
  \label{inv-meas1}
  &
    \xi(x)
    = \sum_{n=1}^{+\infty} \big(\xi^{(n,c)}{\rm
    c}_n(x)+\xi^{(n,s)}{\rm s}_n(x)\big) . 
\end{align}
 Here  $\xi^{(n,\iota)}$, $n=1,2,\ldots$, $\iota\in\{c,s\}$ are
i.i.d. standard real valued Gaussians and 
\begin{align*}
&{\rm c}_n(x):=\sqrt{2}\cos(2\pi 
nx),\quad  {\rm s}_n(x):=\sqrt{2}\sin(2\pi 
nx)\quad x\in\bbT,\,n=1,2,\ldots ,
\\
&
{\rm c}_0(x)\equiv 1.
\end{align*}

\subsection{Main results}

Using   standard results on the stochastic heat
equation, see e.g. \cite{walsh}, 
one can show
that $\big(u_V(t;u)\big)$, the solution of \eqref{011804-26}
satisfying $u_V(0;u)=u$, with $u\in\rH$,  is a continuous trajectory Feller-Markov
process taking values in $H$. 
Denote by $\big({\cal U}^V_t\big)$ {the transition
  probability semigroup associated with the    process} and
by $\big({\cal P}^V_t\big)$  its
 dual  semigroup, defined on $B_b(\rH)$ and $ 
{\cal M}_1\big( {\rm H}\big)$ respectively. 
Given $\mu\in
{\cal M}_1\big( {\rm H}\big)$  we let $\PP_{\mu}:=\mu\otimes \PP$,
defined on $H\times\Om$,
 and
$\EE_{\mu}$ be the corresponding expectation.

{First, we have the following result regarding the existence, uniqueness and the exponential stability of the invariant measure.}  
 \begin{theorem}
   \label{thm012504-26}
  There exists a unique invariant, Borel probability measure
 $\nu_V \in {\cal M}_1\big( {\rm H}\big)$ for
 the solution of \eqref{011804-26}, i.e.
 \begin{equation}
   \label{inv-m}
   \int_{\rH}{\cal U}^V_tF(u) \nu_V(du)=\int_{\rH}F(u) \nu_V(du),\quad
   t\ge0,\,F\in B_b(\rm H).
 \end{equation}
  The measure $\nu_V$ has the following properties:

 \begin{itemize}
   \item[i)] for any $F\in C_b({\rm H})$ and $\mu\in {\cal M}_1\big( {\rm H}\big)$, we have
 \begin{equation}
   \label{AS}
   \lim_{t\to+\infty}\int_{\rH}{\cal U}^V_tF(u) \mu(du)=\int_{\rH}F(u) \nu_V(du),
 \end{equation}
 \item[ii)] for any $p\in[0,+\infty)$   we
have
\begin{equation}
   \label{p-moment}
   \int_{\rH} \|u \|_{H^{-1}(\bT)}^p  \nu_V(du)<+\infty,
 \end{equation}
 \item[iii)] for any $\mu\in
   {\cal M}_1\big( {\rm H}\big)$ for which
 \begin{equation}
   \label{p-moment1}
  u_*:= {\limsup_{t\to0+}}\,\EE_{\mu}\|u_V(t)\|_{{H^{-1}(\bT)}}<+\infty
 \end{equation}
  there exists constant $C >0$, depending only  $u_*$ and
  $\|V\|_{L^2(\bT)}$, and $\lam>0$, depending only on    
  $\|V\|_{L^2(\bT)}$,   such that  
 \begin{equation}
   \label{F-M-stab}
   \begin{split}
     \Big| \EE_{\mu}\Big[u_V(t)(\varphi)\Big]&-\int_{\rH}u(\varphi)
     \nu_V(du)\Big|
     \le C 
  e^{-\lam t}   \|\varphi\|_{H^1(\bT)},\quad
  t\ge0,\,\varphi\in H^1(\bT).
  \end{split}
  \end{equation}
  \end{itemize}
\end{theorem}
This  result  follows from a rather straightforward
 extension of \cite[Lemma 2.2 and Theorem 2.3]{GK}.  We present the
 proof   in Section \ref{ASFM}.

 According to our knowledge, more precise  description of  the invariant measure $\nu_V$ is not known, except
 in the case $V=0$.  
It has been shown in \cite{FQ} that then
  $\nu_0=\nu$, where $\nu$ is the spatial mean-zero white
  noise on the torus, described in Section \ref{sec-gwn}.


Suppose that $F(\cdot)\in   D(\nu)$ - the set of
 densities on  ${\cal M}_1( {\rm H})$ w.r.t. $\nu$.
 Denote by $u_V(t;F)$ the solution of \eqref{011804-26}
 and the initial
 data $u_V(0;F)$ is distributed according to $\mu(du)=F(u)\nu(d u) $. 
Denote by $\mu_{t,V}(F)$ the law of  $u_V(t;F)$. The following result
says that the law    is absolutely continuous with respect to $\nu$.
We can write therefore $\mu_{t,V}(F)={\cal
  P}_t^VF d\nu$, where ${\cal
  P}_t^VF\in D(\nu)$.
  \begin{proposition}
  \label{sec1aC}
 Suppose that $t\ge0$. Then, the probability measure $\mu_{t,V}(F)$ is
 absolutely continuous w.r.t. the measure $\nu$.  
 We denote its respective
 Radon-Nikodym derivative   by ${\cal
  P}_t^VF$. As a result, the family $\big({\cal P}_t^V\big)_{t\ge0}$ forms a strongly continuous
semigroup of linear operators on $L^1(\nu)$ which  satisfies
\begin{equation}
  \label{f-p}
  \begin{split}
    &{\cal P}_t^VG\ge0,\quad\mbox{if }G\ge0,\\
&  \int_{\rm H}{\cal P}_t^VGd\nu=\int_{\rm H} Gd\nu,\quad G\in
L^1(\nu).
\end{split}
\end{equation}
\end{proposition}

Any operator on $L^1(\nu)$ that satisfies   \eqref{f-p}  shall be
called a {\em Frobenius-Perron} (transfer) operator. It is obviously a
contraction on $L^1(\nu)$. The proof of the above result is  presented in Section \ref{secL1sa}.

The main result of the present paper is the following.
\begin{theorem}
  \label{cor022404-26}
  Under the assumption \eqref{V00}, {$\nu_V$ is absolutely continuous with respect to $\nu$, i.e.,}  there exists a unique $\bar
  F_V\in D(\nu)$ such that
   \begin{equation}
    \label{012704-26}
    {\cal P}_{t }^V\bar F_V =\bar F_V,\quad t\ge0.
  \end{equation}
  The following are true: 
  
  \begin{itemize}
  \item[(i)]
 The relative entropy of $\nu_V$ w.r.t. the white noise measure $\nu$ is finite:
  \begin{equation}
    \label{rel-ent}
{\mathbf{H}}_{V}(\nu_V\big|\nu )=\int_{\rH}\bar F_V\log \bar F_V d\nu<+\infty.
\end{equation}

\item[(ii)] Starting from an arbitrary density, the convergence to the invariant measure is in the total variation sense: 
   \begin{equation}
    \label{n072404-26}
    \lim_{t\to+\infty} \|{\cal P}_{t}^VF-\bar
    F_V\|_{L^1(\nu)}=0,\quad \mbox{for all}\,F\in D(\nu).
  \end{equation}
  
  \item[(iii)] $\bar F_V$ belongs to $L^p(\nu)$ for small
    perturbations: there exists a universal constant $C>0$ such that
  \[  \mbox{if $p\in(1,+\infty)$ and $ \|V\|_{L^2(\bT)}< Cp^{-1}$, 
  then $\bar F_V\in L^p(\nu)$.}
  \]
  \end{itemize}
  \end{theorem}
The proof of the result is contained in Section \ref{secL1sb}.

\section{Existence and stability of invariant measure: proof of Theorem \ref{thm012504-26}}

\label{ASFM}

{The goal of this section is to prove Theorem~\ref{thm012504-26}, which states that, for the perturbed stochastic Burgers equation, there exits a unique invariant measure and the equation satisfies  an exponential stability of the form \eqref{F-M-stab}. As a matter of fact, we will prove something much stronger, see \eqref{012304-21z} below, which is the so-called \emph{synchronization} or \emph{one-force-one-solution} principle in the study of the random dynamical systems. It allows us to compare two solutions to the equation driven by the same noise and with different initial conditions, and shows that the difference between the two solutions is exponentially small in time. For random Burgers equation on the torus, this is not surprising and goes back to the early work of Sinai \cite{sinai}. Our proof is a small adaptation of the argument presented in \cite{GK} for the case of $V=0$, which itself was inspired by Sinai's argument and   depends strongly on the compactness of the domain in consideration. For the convenience of readers we sketch it below.}

\subsection{Estimates of the fundamental solution of \eqref{e.she}}
For any $t>s$ and $x,y\in\bT$, define $\mathcal{Z}_{t,s}^V(x,y;\om)$ as the solution to 
\begin{equation}\label{e.defZts}
\begin{aligned}
&d \mathcal{Z}_{t,s}^V(x,y)=\big[\tfrac12\Delta_x
\mathcal{Z}_{t,s}^V(x,y)+\mathcal{Z}_{t,s}^V(x,y)V(x)\big]dt+\mathcal{Z}_{t,s}^V
(x,y)dW(t,x), \quad\quad t>s,\\
&\mathcal{Z}_{s,s}^V(x,y)=\delta(x-y).
\end{aligned}
\end{equation}
In other words, $\mathcal{Z}_{t,s}^V(x,y)$ is the Green's function of the
SHE, or the propagator from $(s,y)$ to $(t,x)$.  We shall   write  $\mathcal{Z}_{t,s}(x,y)=\mathcal{Z}_{t,s}^0(x,y)$.

Crucial for our argument is the following.
\begin{lemma}\label{l.ulbd}
For any $p\in[1,+\infty)$, there exists $C>0$ only depending on $p,V$ such that 
\begin{equation}\label{e.ulbd}
\EE\Big[\big(\inf_{x,y\in\bT}
\cZ_{t,t-1}^V(x,y)\big)^{-p}\Big]+\EE\Big[\big(\sup_{x,y\in\bT^d}\cZ_{t,t-1}^V(x,y)\big)^p\Big]\leq
C.
\end{equation}
\end{lemma}
\proof The proof is based on an application of the Cameron-Martin formula, which incorporates the perturbation $\xi\mapsto \xi+V$: we can write
 \begin{align*}
  & \EE\Big[\big(\inf_{x,y\in\bT}
   \cZ_{t,t-1}^V(x,y)\big)^{-p}\Big]\\
  &
     =\EE\Big[\big(\inf_{x,y\in\bT}
 \cZ_{t,t-1}(x,y)\big)^{-p}\exp\left\{\la W(t)- W(t-1),V\ra-\frac12\|V\|_{L^2(\bT)}^2\right\}\Big].
\end{align*}
Then, using H\"older inequality, the right hand side can be estimated by
\begin{align*}
 &  \EE\Big[\big(\inf_{x,y\in\bT}
\cZ_{t,t-1}(x,y)\big)^{-p'}\Big]^{p/p'}  \exp\left\{ \frac{q-1}{2}\|V\|_{L^2(\bT)}^2\right\}
\end{align*}
for any $p'>p$, with $1/q+p/p'=1$. Using \cite[Lemma 4.1]{GK} we
conclude that there exists $C>0$ such that
\begin{equation}\label{e.ulbd1}
\EE\Big[\big(\inf_{x,y\in\bT}
\cZ_{t,t-1}^V(x,y)\big)^{-p}\Big] \leq
C.
\end{equation}
The argument for the estimate of the $p$-th moment of the supremum is analogous.
  \qed

\subsection{Endpoint distribution of directed polymer}

{The argument presented in \cite{GK} does not  directly apply to the process $(u_V(t))$. Instead the focus there was the closely related process given by $R^{-1}(u_V(t))$ with $R^{-1}$ defined in \eqref{012804-26-1}. This is the so-called endpoint density of directed polymer in a random environment, }
 see \cite[Chapter 8]{comets}.

For a directed polymer of length $t$, with the starting point
distributed according to density $\rho\in D_+(\bT)$, its endpoint distribution is defined as  
\begin{equation}\label{e.defPo}
\rho^V(t,x; \rho)=\frac{\cZ^V(t,x; \rho)}{\int_{\bT} \cZ^V(t,x'; \rho)dx'},
\end{equation}
{where 
\[
\cZ^V(t,x;\rho):=\int_{\bT}\cZ^V_{t,0}(x,y)\rho(y)dy.
\]
Sometimes $\rho^V$ is also referred to as the projective process, as it can be viewed as the projection of $\cZ^V$ onto the $L^1(\bT)-$unit ball.} 
Throughout the paper, we will omit the dependence of either $\cZ^V$, or
$\rho^V$ on the initial data   when there is no danger of
confusion.  
We have $\rho^V(t)\in   D_+(\bT)$ for any $t>0$.

It has been shown in \cite[Lemma 2.2]{GK} that the process $\big(\rho^V(t)\big)$ is
Markov. The proof shown there deals with the case $V=0$, but can be
adapter, modulo obvious modification, to the case of  $V\in C^2(\bT)$
considered here.
Denote by $\big(\tilde \U^V_t\big)$ and $\big(\tilde \cP^V_t\big)$
the transition probability semigroup and its adjoint 
corresponding to $\big(\rho^V(t)\big)$. The following result is a
consequence of the regularity result for   $\big({\cal
    \cZ^V}(t)\big)_{t\geq0}$.
\begin{lemma}\label{l.feller} For any $t\ge0$, 
  $$
  \tilde \U^V_t\big(C_b(  D_+(\bT))\big)\subset C_b( 
  D_+(\bT)).\quad
  $$
\end{lemma}
 
 Having established Lemma \ref{l.ulbd}, the proofs of
  Theorem 2.3,  Proposition 4.7 and Lemma 4.10 of \cite{GK} can be repeated, almost
 verbatim, and we can formulate the following result. 
 \begin{theorem}
\label{cor020104-21}
Suppose that $F\in C_b (  D_+(\bT))$ and $\mu\in {\cal
  M}_1( D_+(\bT))$, then there exists a unique measure  $\pi_\infty\in {\cal
  M}_1( D(\bT))$  such that
\begin{equation}
\label{040104-21}
\lim_{t\to+\infty}\int_{ D_+(\bT)}\tilde {\cal
    U}_t^VFd\mu= \int_{  D_+(\bT)}F d\pi_\infty .
\end{equation}
The measure $\pi_\infty$ is invariant under  $\big(\tilde
\U^V_t\big)$.

In addition,  for any $\delta>0$, $p\in[1,+\infty)$
there exist constants $C,\lam>0$, depending only on
$\|V\|_{L^2(\bT)}$, $\delta$ and $p$,
such that
\begin{equation}
\label{012304-21z}
\sup_{\rho_1,\rho_2\in
  D_+(\bT)}\EE\big\|\rho^V(t;\rho_1)-\rho^V(t;\rho_2)\big\|^p_\infty\leq
C e^{-\lambda t},\quad t\ge \delta.
\end{equation}

Furthermore,    there exists $C>0$, depending
only on $\|V\|_{L^2(\bT)}$, $\delta$ and $p$ such that  
\begin{equation}
\label{013004-26}
\sup_{t\ge \delta,\,{\mu\in {\cal
    M}_1(\bT)}}\EE_{\mu}\Big[{\|\rho^V(t;\mu)\|_\infty^p+\|\rho^V(t;\mu)^{-1}\|_\infty^{p}}\Big]\le C.
\end{equation}
In  consequence, for any $p\ge1$,
\begin{equation}
\label{040104-21b}
\int_{ D_+(\bT)}\big(\|\rho\|^p_\infty +{\|\rho^{-1}\|^{p}_\infty}\big) \pi_\infty (d\rho)<+\infty.
\end{equation}
\end{theorem}

\medskip

\subsection{Proof of Theorem~\ref{thm012504-26}}

Note that, given $u\in\rH$ we can write
\begin{equation}
  \label{Ru}
  u_V(t;u)=R\big(\rho^V(t;R^{-1}(u)\big),
\end{equation}
where the operator $R$ was defined in \eqref{012804-26}. 
The transition probability semigroup and its dual satisfy
\begin{align*}
&\U^V_tF(u)=\big(\tilde \U^V_tF\circ R\big)(R^{-1}(u)),\quad u\in\rH\quad
\mbox{and}\\ 
&
\cP^V_t \mu=\tilde \cP^V_t (\mu R)R^{-1},\quad t\ge0,\mu\in {\rm M}_1(\rH).
\end{align*}
A direct consequence of Lemma \ref{l.feller} is the following.
\begin{lemma}\label{l.feller1} For any $t\ge0$
  $$
   \U^V_t\big(C_b(\rH)\big)\subset C_b(\rH).
  $$
\end{lemma}

\subsubsection{Proof of part (i)}

Define
$
\nu_V:=\pi_\infty R^{-1}
$ - a Borel probability measure on $\rH$.
We claim that for any  {$F\in C_b(\rH)$}, equality \eqref{AS}
holds. Together with Lemma \ref{l.feller1} they also imply \eqref{inv-m}.
To prove  \eqref{AS} it suffices to show 
\begin{equation}
   \label{AS1}
   \lim_{t\to+\infty}\EE F\big(u_V(t;u)\big) =\int_{\rH}F 
   d\nu_V ,\quad \mbox{ for any }u\in \rH.
 \end{equation}
 Given $\eps\in(0,1)$ define
 $$
D_\eps:=[\rho\in D_+(\bT):\,\eps\le \rho\le \eps^{-1}].
$$
The set is closed both in $D_+(\bT)$ and $ D(\bT)$. The two metrics
${\rm d}_{\rm D}$ and ${\rm d}_\infty$ are equivalent on $D_\eps$.

Choose an arbitrary $\sigma>0$.
Recalling \eqref{Ru} we can write 
\begin{align*}
  &
    \Big|\EE F\big(u_V(t;u)\big) -\int_{\rH}F 
   d\nu_V \Big|=\Big|\EE F \circ R\big(\rho^V(t;R^{-1}(u))\big) -\int_{D_+(\bT)}F\circ R
  d\pi_\infty\Big|\\
  &
    \le \Big|\EE F \circ R\big(\rho^V(t;R^{-1}(u))\big) -\EE \Big[F
    \circ R\big(\rho^V(t;R^{-1}(u))\big),\,\rho^V(t;R^{-1}(u))\in D_\eps\Big]\Big|
  \\
  &
   + \Big|\EE \Big[F
    \circ R\big(\rho^V(t;R^{-1}(u))\big),\,\rho^V(t;R^{-1}(u))\in D_\eps\Big]-\int_{D_+(\bT)}F\circ R
  d\pi_\infty\Big|.
\end{align*}
Denote the terms on the utmost right hand side by $I_1(t)$ and
$I_2(t)$ repectively. By virtue of \eqref{013004-26} we can choose
$\eps$ so small that $I_1(t)<\sigma$ for all $t>1$. Concerning the
other term,
we write $I_2(t)\le I_{2,1}(t) +I_{2,2}(t)$, where
 \begin{align*}
   &
     I_{2,2}(t):=\Big|\EE \Big[(F
    \circ R)^*\big(\rho^V(t;R^{-1}(u))\big) \Big]-\int_{\mm{D_+}(\bT)}F\circ R
     d\pi_\infty\Big|,\\
   &
     I_{2,1}(t):=\Big|\EE \Big[(F
    \circ R)^*\big(\rho^V(t;R^{-1}(u))\big),
     \,\rho^V(t;R^{-1}(u))\not\in D_\eps \Big] \Big|.
 \end{align*}
 Here $(F
    \circ R)^*$ is a continuous extension of $F
    \circ R_{|D_\eps}$ to $  D(\bT)$ such that $\|(F
    \circ R)^*\|_\infty\le \|F
     \|_\infty$. Using again  \eqref{013004-26}  we can choose
     $\eps$ so small that $I_{2,1}(t)<\sigma$ for all $t>1$. Hence,
     applying Theorem \ref{cor020104-21} we conclude that
   \begin{align*}
  &
   \limsup_{t\to+\infty} \Big|\EE F\big(u_V(t;u)\big) -\int_{\rH}F 
    d\nu_V \Big|\le 2\sigma+  \lim_{t\to+\infty}  I_{2,2}(t)\\
     &
       =2\sigma +\Big|\int_{ D(\bT)}(F\circ R)^*
     d\pi_\infty-\int_{D_+(\bT)}F\circ R
     d\pi_\infty\Big|\le {2\sigma}+ 2\|F
     \|_\infty\pi_\infty(D_\eps^c).
\end{align*}  
Since  
$
\pi_\infty (D_+(\bT))=1
$
we can choose $\eps$ so small, so that
the last term on the right hand side is also less than $\sigma$. This
ends the proof of \eqref{AS1}, concluding also the proof of part (i) of
the theorem.

\subsubsection{Proof of part (ii)} To conclude \eqref{p-moment}, note
that by  
\eqref{012804-26}
we have
\begin{align*}
   &
    |u(\varphi)|=|\langle
     \log\rho,\nabla\varphi\ra_{L^2(\bT)}| \le 
 \|\varphi\|_{H^1(\bT)}  \| 
   \log\rho\|_\infty\quad\mbox{for any   $\varphi\in H^1(\bT)$} .
\end{align*}
Therefore
\begin{align*}
   &
     \int_{\rH}\|u\|_{H^{-1}(\bT)}^p \nu_V(du) = \int_{\rH}{\sup_{\varphi:\|
     \varphi\|_{H^{1}(\bT)}=1}}|u(\varphi)|^p \nu_V(du) \le \int_{D_+(\bT)}\| 
   \log\rho\|_\infty ^p \pi_\infty(d\rho) .
 \end{align*}
Using an elementary bound $|\log x|\le x+\frac{1}{x}$,
$x\in(0,+\infty)$, together with \eqref{040104-21b} we conclude \eqref{p-moment}.
 
\subsubsection{Proof of part (iii)}

Let $\pi:=\mu R^{-1}$. Thanks to \eqref{p-moment} and
\eqref{p-moment1} there exists $C,\delta>0$ such that
\begin{align}
  \label{043004-26}
        \sup_{t\in[0,\delta]}\Big| \EE_{\mu}\Big[u_V(t)(\varphi)\Big]-\int_{\rH}u(\varphi)
     \nu_V(du)\Big|\le C\|\varphi\|_{H^1(\bT)}. 
\end{align}
For $t\ge \delta$ we can write
\begin{align}
  \label{023004-26}
   &
     \Big| \EE_{\mu}\Big[u_V(t)(\varphi)\Big]-\int_{\rH}u(\varphi)
     \nu_V(du)\Big|\\
   &
     =  \Big|\int_{D_+(\bT)}\la \log \rho ,\nabla\varphi\ra_{L^2(\bT)}\
     \pi_\infty(d\rho)- \EE_{\pi}\Big[\la \log
     \rho^V(t),\nabla\varphi\ra_{L^2(\bT)}\Big]\Big|.\notag
\end{align}
{Here $\EE_\pi$ is interpreted as the expectation on $(\rho^V(t))$ with $\rho^V(0)$ distributed according to $\pi$.}  
Thanks to \eqref{013004-26} for we have
$
\sup_{t\ge \delta}\EE_{\mu}\Big[|u_V(t)(\varphi)|\Big]<+\infty.
$ 
 
%
By mean value theorem,  we have
 \begin{align}
   \label{033004-26}
 \Big| \log  \rho_2- \log
                 \rho_1 \Big|\le 
                \Big(\frac{1}{\rho_1}+\frac{1}{\rho_2}\Big) |\rho_2- \rho_1|,
 \end{align}
using  which  the right hand side of
\eqref{023004-26} can be estimated by
\begin{align*}&
                \|\varphi\|_{H^1(\bT)}\int_{\rH^2}\pi_\infty(d\rho_1)
                \pi(d\rho_2)  \EE\big\|\log \rho^V(t;\rho_1) -\log
                \rho^V(t;\rho_2) \big\|_{\infty}\\
              &
       \le  \|\varphi\|_{H^1(\bT)}\sup_{\rho_1,\rho_2\in D_+(\bT)}
                \Big\{\EE
                \Big\|\frac{1}{\rho_1}+\frac{1}{\rho_2}\Big\|_\infty^2
                \Big\}^{1/2}\sup_{\rho_1,\rho_2\in D_+(\bT)}
                \Big\{\EE\big\|  \rho^V(t;\rho_1) - 
                \rho^V(t;\rho_2) \big\|_{\infty}^2     \Big\}^{1/2}
\end{align*}
and   an application of \eqref{012304-21z} and \eqref{013004-26}
allows us to show \eqref{F-M-stab} for $t\ge\delta$. This combined
with \eqref{043004-26} ends the
proof of part (iii).\qed

\section{Discretization and functional inequalities}

\label{Disc}

{One of the main difficulties in dealing with equations of the form
  \eqref{011804-26} is the singularity of the noise. Indeed,  to make
  sense of the stochastic Burgers equation without appealing to the Cole-Hopf transformation relies on the tools from the theory of singular SPDE \cite{GP17,Ha13}. For our purpose of studying the properties of the  invariant measure, it seems natural to study the evolution of the law of $u_V(t)$, which should be governed by a Fokker-Planck type equation thereby involves the   generator of the corresponding Markov process. For this singular SPDE, it is already a highly nontrivial task to construct a domain of the generator \cite{GP20}. Although it might be possible to work directly in the setting of \cite{GP20}, we choose to approximate the problem by a finite dimensional SDE, where one can see rather clearly the calculations involving the generator. This section is devoted to such approximations.}

{On the discrete level, we study the evolution of the relative entropy and $L^p-$norm of the density, and obtain estimates that are uniform in the size of discretization, by applying functional inequalities associated with the underlying Gaussian measure.}

\bigskip

Suppose $\eta\in C^\infty(\bbR)$ is  a non-negative function supported on
$[-1/4,1/4]$, such that \begin{align*}
                          \eta(-x)=\eta(x),\quad\mbox{and}\quad
                          \int_{\bbR}\eta(x)dx=1.
                        \end{align*}
Its Fourier transform
$
{\cal F}(\eta)(\xi)=\int_{\bbR}e^{-2\pi i\xi x}\eta(x)dx$, $\xi\in\bbR$
is even and real valued.

Throughout this section, we fix some small $\eps>0$  (say
  $\eps<1/100$) and large positive integer $N$ (say  $N\ge 100$). Define
\begin{equation}
  \label{040405-26}
  \eta^{(\eps)}(x):=\eps^{-1}\eta(x/\eps)\quad\mbox{ and}\quad 
  \eta_2^{(\eps)}(x):=\eta^{(\eps)}\star \eta^{(\eps)} (x).
  \end{equation}
Then,
\begin{align}
   \label{050405-26}
  &{\cal F}\big(\eta^{(\eps)}\big)(\xi)={\cal F}(\eta) (\eps\xi)\quad\mbox{and}\quad
      {\cal F}\big(\eta^{(\eps)}_2\big) (\xi)=\big({\cal F}(\eta) (\eps\xi)\big)^2.
\end{align}

We
consider $1$-period extensions of both $\eta^{(\eps)}$ and
$\eta_2^{(\eps)}$, so that the  functions are defined on $\bbT$, which we denote by
the same symbols.
The Fourier coefficients of $\eta^{(\eps)}$ and $\eta^{(\eps)}_2$ are given by 
\begin{equation}
  \label{ft}
  \begin{split}
  &\hat\eta^{(\eps)}(k)=\int_{\bT}\eta^{(\eps)}(x)e^{-2\pi i
    kx}dx=\int_{\bT}\eta^{(\eps)}(x)\cos(2\pi kx)dx\\
  &
  =\int_{\bbR}\frac{\sin(\pi\xi) }{\pi\xi}{\cal F}(\eta)(\eps(\xi+k) )d\xi,\quad\mbox{and},\\
    &\hat\eta^{(\eps)}_2(k)=\int_{\bbR}\frac{\sin(\pi\xi) }{\pi\xi}\big({\cal F}(\eta)\big)^2(\eps(\xi+k)) d\xi, \quad
  k\in\bbZ.
\end{split}
\end{equation}
Obviously
$
\hat\eta^{(\eps)}(0)=\hat\eta^{(\eps)}_2(0)=1.
$

\subsection{Approximation scheme}

Suppose that   $\bT_N$ is a discrete torus of size $N$,
i.e. $\bT_N:=\{0,\ldots,N-1\}$, with the addition operation modulo
$N$. For any $f:\bbT_N\to\bbR$, we let
\begin{align*}
&\nabla f_j=f_{j+1}-f_j,\quad  \nabla^* f_j=f_{j-1}-f_j,\quad\mbox{and}\\ 
&
\Delta f_j=\mm{-}\nabla^* \nabla f_j=f_{j+1}+f_{j-1}-2f_j,\quad  j\in\bT_N  .
\end{align*}
In what follows we let
$$
\Om_N:=\big[(u_j)_{j\in\bT_N}\in\bbR^{\bT_N}:\,\sum_{j\in\bT_N}u_j=0\big].
$$

Let 
\begin{align}\label{e.defalpha2}
   &
     \alpha(j)=\frac{1}{N}\Big[\eta^{(\eps)}\Big(\frac{j}{N}\Big)+\mm{c_{N,\eps}}\frac{\delta_{0,j}}{N^{100}}\Big],
     \quad \alpha_2(j)= \alpha\star \alpha(j)= \sum_{j'\in\bT_N}\al(j-j')\al(j') .
\end{align}
{ Here $c_{N,\eps}\in(0,1)$ is some constant to be determined. For simplicity, we have omitted the dependence of $\alpha$ and $\alpha_2$ on $\eps,N$. One may think of $\alpha_2$ as an approximation of $\eta^\eps\star \eta^\eps$, as $N\to\infty$, because for any continuous functions $f:\bT\to\R$ we have 
\begin{equation}\label{e.5291}
\begin{aligned}
&\sum_{j\in \bT_N} \alpha_2(i-j)f(\frac{j}{N})\approx \frac{1}{N^2} \sum_{j,j'\in\bT_N} \eta^{(\eps)}(\frac{i-j-j'}{N})\eta^{(\eps)}(\frac{j'}{N})f(\frac{j}{N})\\
&\approx \int_{\bT} \eta^\eps\star\eta^\eps(\frac{i}{N}-x)f(x)dx.
\end{aligned}
\end{equation}
}
 
{First, we note that the  matrix $[\alpha_2(j-j')]_{j,j'\in\bT_N}$ is non-negative
definite, as 
\begin{align*}
  &\sum_{j,j'\in\bT_N}\alpha_2(j-j')\xi_j\xi_{j'}^\star=\frac{1}{N}\sum_{k\in\hat\bT_N}\hat
    \alpha^2(k) |\hat\xi(k)|^2\geq 0, 
\end{align*}
with  $\hat\xi(k)$ the Fourier transform of   the sequence
$\big(\xi_j\big)_{j\in\bT_N}$, i.e.
$$
\hat \xi(k)=   \sum_{j\in\bT_N}\exp\left\{-i2\pi  j k
\right\} \xi(j), \quad k\in  \hat\bT_N :=\big\{\frac{\ell}{N},\,\ell =0,\ldots,N-1\big\}.
$$  
Note that, since $\al(-j)=\al(j)$, we have $\hat \alpha(k)\in\bbR$. Now we claim that one can choose the constant $c_{N,\eps}\in(0,1)$ so that $\hat{\alpha}(k)\neq0$ for all $k\in \hat\bT_N$, and as a result the matrix  $[\alpha_2(j-j')]_{j,j'\in\bT_N}$ is positive definite. This is because one can write 
\[
\hat{\alpha}(k)=\sum_{j\in\bT_N}\exp\left\{-i2\pi  j k
\right\} \alpha(j)=\sum_{j\in\bT_N}\exp\left\{-i2\pi  j k
\right\} \frac{1}{N}\eta^{(\eps)}\Big(\frac{j}{N}\Big)+ \frac{c_{N,\eps}}{N^{101}}.
\]
Denote the first term on the r.h.s. by $\beta(k)$, and let $(1)_k$ be the vector whose coordinates are  $1$, one can continuously change the value of $c_{N,\eps}$ so that $(\hat{\alpha}(k))_k=(\beta(k))_k+\frac{c_{N,\eps}}{N^{101}}(1)_k$ has no zero coordinates.}


 Define   the process
   $
   \Big(u_{N,V}\big(t,\frac{j}{N}\big)\Big)_{j\in\bT_N}
   $
  as the solution of the following stochastic differential equation 
\begin{equation}
\label{burgers.dSVN}
\begin{split}
&
d u_{N,V}\big(t,\frac{j}{N}\big)= \Big\{\frac{ 1}{2}\big(N^2\Delta\big) u_{N,V}\big(t,\frac{j}{N}\big)+\frac{1}{6}
(N\nabla) \alpha_2\star
G\Big (u_{N,V}\big(t,\frac{j}{N}\big)\Big)\Big\}dt\\
&
+ (N \nabla V^{(\alpha_2)}) \big(\frac{j}{N}\big) dt+ d( N\nabla)
W^{(\alpha)}\big(t,\frac{j}{N}\big),\quad j\in \bT_N,\\
&
\mbox{where }\quad W^{(\alpha)}\big(t,\frac{j}{N}\big)=\sqrt{N} w_j^{(\al)}.
\end{split}
\end{equation}
Here
\begin{equation}
\label{burgers.dSV}
\begin{split}
&
 V^{(\alpha_2)}\big(\frac{j}{N}\big)=V\star\al_2(j)=\sum_{j'\in\bT_N}\al_2(j-j')
V\Big(\frac{j'}{N}\Big),\\
&
G (u)(j):= u_j^2+u_j u_{j-1}+u_{j-1}^2 ,\quad u=(u_j)_{j\in\bT_N},\notag\\
&w_j^{(\alpha )}(t):=\alpha\star w_j(t), \notag
\end{split}
\end{equation}
and   $(w_j(t))_{t\ge0}$, $j\in\bbT_N$ are  independent, one dimensional,
standard Brownian motions defined over the probability space
$(\Om,{\cal F},\PP)$. 

{Let us pause to explain the choice of the scaling here. For fixed $\eps>0$, as $N\to\infty$, we expect that $N^2\Delta,N\nabla$ to converge to the continuous $\Delta,\nabla$ respectively. Secondly, by \eqref{e.5291}, $\alpha_2\star G$ and $V^{(\alpha_2)}$ converges to the corresponding continuous counterpart. For the driving noise, since $dW^{(\alpha)}$ has the spatial covariance function given by $N\alpha_2$ which converges to $\eta^\eps\star \eta^\eps$, one may expect that its limit is given by a spacetime white noise mollified by $\eta^\eps\star \eta^\eps$. As a result, on the formal level, the large $N-$limit of \eqref{burgers.dSVN} is given by the SPDE \eqref{eq:SBEV} below.}

The generator of   diffusion \eqref{burgers.dSVN} is given by 
\begin{align}
  \label{L-alN}
  & {\cal  L}^{(\alpha)}_{N,V}f(u)={\cal L}^{(\alpha)}_N f(u) +{\cal A}_{N,V}f(u),\quad\mbox{where} \quad {\cal L}^{(\alpha)}_N={\cal S}^{(\alpha)}_N+{\cal A}^{(\alpha)}_N, \\
 & {\cal   A}_N^{(\alpha)}f(u) =
   \frac{1}{6}\sum_{j\in\bT_N}(N\nabla)\alpha_2\star
    G(u)\Big(\frac{j}{N}\Big) \partial_{u_j}f(u),\quad m=1,2 \quad\mbox{and}\notag\\
 &
 G\big (u\big) \Big(\frac{j}{N}\Big):=
   u^2\Big(\frac{j}{N}\Big) +u\Big(\frac{j}{N}\Big)
   u\Big(\frac{j-1}{N}\Big)                +u^2\Big(\frac{j-1}{N}\Big) , \notag\\
   &
     {\cal
     A}_{N,V}f(u)= \sum_{j\in\bT_N} (N \nabla V^{(\alpha_2)}) \big(\frac{j}{N}\big) \partial_{u_j}f(u), \notag
\end{align}
and 
\begin{align}
  \label{gen-SN}
 & {\cal  S}^{(\alpha)}_Nf(u)= -\sum_{j,j'\in\bT_N}\frac{ 1}{2} (
   {N^3}\Delta \al_2)\big({j-j'}\big)
   \partial_{u_j}\partial_{u_{j'}}f(u)+\frac{1}{2}\sum_{j\in\bT_N}(N^2\Delta) u
   \big(\frac{j}{N}\big)\partial_{u_j}f(u).
\end{align}
We   embed $\Om_N$ into $C(\bT)$ by the linear interpolation of the
nodal points $\Big(\frac{j}{N} ,u\big(\frac{j}{N}\big)\Big)$,
$j\in\bT_N$. We can treat the evolution, described on
$\Om_N$ by \eqref{burgers.dSVN}, as  taking place  in $C(\bT)$.
The resulting process shall be denoted then by  $
\Big(u_{N,V}(t,x)\Big)_{(t,x)\in[0,+\infty)\times \bbT}$.  

In the
unperturbed case $V=0$ we omit writing the superscript  (or subscript) in our notation,
e.g. we write $u_{N}(t,x) =u_{N,0}(t,x)$.

For $u=\Big(u\big(\frac{j}{N}\big)\Big)_{j\in\bT_N}\in\Om_N$ we define the Gaussian density
\begin{equation}
  \label{CNj}
  \begin{split}
&\nu_{S}^{(N)} (du):=\frac{1}{Z_N}\exp\left\{-\frac{1}{2}{\cal
    C}_N^{-1}u\cdot u\right\}  du ,\quad
\big({\cal
    C}_N\big)_{j,j'}:=N\alpha_2(j-j'),\quad j,j'\in\bT_N.
\end{split}
\end{equation}
Here $Z_N$ is the normalizing factor. We
denote  the push-forward of $\nu_{S}^{(N)} (du)$   on $C(\bT)$ by the
same symbol.

According to \cite[Proposition 2.3]{FQ}, the measure $\nu_{S}^{(N)}(du)$ is
invariant under the dynamics described by \eqref{burgers.dSVN} with $V=0$. We also
have  the following result, which states that ${\cal   S}^{(\alpha)}_N$ and ${\cal
    A}^{(\alpha)}_N$ are the symmetric and anti-symmetric part of the generator when $V=0$, see \cite[Lemma 2.2]{FQ}. 
\begin{proposition}
  \label{prop023103-26N}
  For any $f,g\in C^2_b(\Om_N)$ we have
\begin{align}
  \label{012703-26N}
  &\int_{\Om_N}{\cal   S}^{(\alpha)}_Nf(u)
    g(u)\nu_{S}^{(N)}(du)=\int_{\Om_N}f(u) {\cal
    S}^{(\alpha)}_Ng(u)\nu_{S}^{(N)} (du) 
\end{align}
and
\begin{align}
  \label{012703-26bN}
  & \int_{\Om_N}{\cal
    A}^{(\alpha)}_Nf(u) g(u)\nu_{S}^{(N)} (du) 
   =- \int_{\Om_N}f(u){\cal   A}^{(\alpha)}_Ng(u)
   \nu_{S}^{(N)}(du).
\end{align}
\end{proposition}

The operator ${\cal   S}^{(\alpha)}_N$ can be rewritten in the divergence form
\begin{align}
  \label{div-fN}
  {\cal
  S}^{(\alpha)}_Nf(v)= -\frac{1}{2}\exp\Big\{\frac12 {\cal
    C}^{-1}_Nv\cdot v\Big\}({N^3}\Delta{\mathbf \alpha_2}D_v)\cdot\Bigg(\exp\Big\{-\frac12 {\cal
    C}^{-1}_Nv\cdot v\Big\}D_vf(v)\Bigg),
\end{align}
with $D_v=(\partial_{v_0},\ldots, \partial_{v_{N-1}})$. Hence, for any
$f,g\in C_b^2(\Om_N)$,
\begin{align}
  \label{011104-26}
 \int_{\Om_N} {\cal
  S}^{(\alpha)}_Nf(v)g(v) \nu_{S}^{(N)} (dv) = \int_{\Omega_N}\frac{1}{2}
  {N^3}\Delta{\mathbf \alpha_2}D_vf(v)\cdot D_vg(v) \nu_{S}^{(N)} (dv) .
\end{align}

Denote by $Q_{N,V}$ the law of  $\big(u_{N,V}(t)\big)$ on  {${\cal C}=C([0,\infty),C(\bT))$, }
assuming that the distribution of initial data $u_{N,V}(0)$ is given
by $\nu_{S}^{(N)} (du)$. 


\subsection{Evolution of density function and entropy estimate}

\label{discr}

 {In this section, we assume that the initial distribution of
  $u_{N,V}(0)$ is absolutely continuous with respect to $\nu_S^{(N)}$,
  the invariant measure for the unperturbed dynamics, and study the
  evolution of the density of the law of $u_{N,V}(t)$. In particular, we investigate the evolution of the relative entropy with respect to $\nu_S^{(N)}$ and derive a differential inequality which leads to the uniform bound on the entropy.} 

We suppose that at $t=0$ the initial distribution of $
u_{N,V}(0)$ is given by
\begin{equation}
  \label{muN0}
\mu_N(0,du)=f_N (u)\nu_{S}^{(N)} (du),
\end{equation}
where $f_N:\R^N\to\R$ is a smooth density.
Let
\[
\mu_N^{(V)}(t,du)=\Q^{(N)}_V(t,u;
 f_N) \nu_{S}^{(N)} (du), 
\] where  $\Q^{(N)}_V(t,u;
 f_N)$ is  
 the  density of  the distribution of $
u_{N,V}(t)$  with respect to
 $\nu_{S}^{(N)}$.

 Recall that the generator of the diffusion is given by $\mathcal{L}_{N,V}^{(\alpha)}=\mathcal{L}_N^{(\alpha)}+\mathcal{A}_{N,V}$ in \eqref{L-alN}.  
The density $\Q^{(N)}_V(t,u; f_N )$  satisfies the Fokker-Planck equation
\begin{equation}
  \label{FPN}
  \begin{split}
&\partial_t \Q^{(N)}_V(t,u;
 f_N) = \Big({\cal L}^{(\alpha)}_{N,V}\Big)^\star \Q^{(N)}_V(t,u;
 f_N) \\
 &
 = \Big( {\cal L}^{(\alpha)}_N\Big)^\star \Q^{(N)}_V(t,u;
 f_N) + {\cal A}_{N,V}^\star  \Q^{(N)}_V(t,u;
 f_N) ,
\end{split}
\end{equation}
where the adjoints are computed w.r.t. the Gaussian measure
$\nu_{S}^{(N)}   $.

The relative entropy of $\mu_N^{(V)}(t)$ with respect
to $\nu_{S}^{(N)}  $ is defined by 
\begin{equation}
  \label{eq:7-1aN}
 \begin{split}
&   {\mathbf{H}}_N^{(V)}(t;f_N) := \int_{\bbR^N}  \Q^{(N)}_V(t,u;
 f_N)  \log \Q^{(N)}_V(t,u;
 f_N)  
 \nu_{S}^{(N)} (du)\\
 &
 = \EE_{{\mu_N(0)}} \big[\log   \Q^{(N)}_V(t,u_{N,V}(t);
 f_N)\big] .
 \end{split}
\end{equation}
 Its evolution satisfies
  \begin{equation}\label{e.1281aN}
    \begin{split}
     & \frac{d {\mathbf{H}}_N^{(V)}(t;f_N) }{dt}= \int_{\Om_N}  {({\cal L}^{(\alpha)}_{N})^*}   \Q^{(N)}_V(t,u;
 f_N)  \log \Q^{(N)}_V(t,u;
 f_N)  
 \nu_{S}^{(N)} (du) \\
 &
 +\int_{\Om_N}  {\cal A}_{N,V}^\star \Q^{(N)}_V(t,u;
 f_N)  \log \Q^{(N)}_V(t,u;
 f_N)   
   \nu_{S}^{(N)} (du).
\end{split}
  \end{equation}

For the second term on the right hand side of \eqref{e.1281aN} the
following formula holds.
\begin{lemma}\label{l.l1aN}
  We have
  \begin{equation}
    \label{022803-26N}
\begin{aligned}
   & \int_{\Om_N}  {\cal A} _{N,V} ^\star \Q^{(N)}_V(t,u;
 f_N)  \log \Q^{(N)}_V(t,u;
 f_N)   
 \nu_{S}^{(N)} (du)\\
 &
 =  \int_{\Om_N} \Q^{(N)}_V(t,u;
 f_N) {\frak l}_{N,V}(u)
 \nu_{S}^{(N)} (du)= \EE_{{\mu_N(0)}} {\frak
  l}_{N,V}\Big(u_{N,V} (t ) \Big) .
\end{aligned}
\end{equation}
Here
\begin{equation}
  \label{lN}
  {\frak l}_{N,V}(u):=\frac{1}{N}\sum_{j\in\bT_N} u \Big(\frac{j}{N}\Big) N\nabla
  V\Big(\frac{j}{N}\Big).
  \end{equation}
\end{lemma}

\begin{proof}
 From \eqref{L-alN} we can write
  \begin{align*}
     &\int_{\Om_N}  {\cal A}_{N,V}^\star \Q^{(N)}_V(t,u;
 f_N)  \log \Q^{(N)}_V(t,u;
 f_N)   
       \nu_{S}^{(N)} (du)\\
     &
       = \int_{\Om_N} \Q^{(N)}_V(t,u;
 f_N)  {\cal A} _{N,V} \Big(\log \Q^{(N)}_V(t;
 f_N) \Big)(u)  
       \nu_{S}^{(N)} (du)
  \\
    &
       = \frac{1}{Z_N} \int_{\Om_N}  \sum_{j\in\bT_N}   N\nabla
      V^{(\al_2)}(\frac{j}{N})  \partial_{u_j       }\Q^{(N)}_V(t,u;
      f_N)  \exp\Big\{-\frac12 {\cal
      C}^{-1}_Nu\cdot u\Big\}  du.
\end{align*}

      Integrating by parts we obtain that the utmost right hand side equals (recall that $\mathcal{C}_N=N\alpha_2$)
      \begin{align*}
    &
       \frac{1}{ Z_N} \int_{\Om_N}
      \sum_{j\in\bT^N}  N\nabla V^{(\al_2)}\Big(\frac{j}{N}\Big)  \Q^{(N)}_V(t,u;
      f_N) ((N\alpha_2)^{-1}u) \Big(\frac{j}{N}\Big)        \exp\Big\{-\frac12 {\cal
      C}^{-1}_Nu\cdot u\Big\}   du\\
    &
      = \int_{\Om_N} \Q^{(N)}_V(t,u;
 f_N)\Big[ \frac{1}{N}\sum_{j\in\bT^N}   u \Big(\frac{j}{N}\Big)  N\nabla V\Big(\frac{j}{N}\Big)      \Big]     \nu_{S}^{(N)} (du)
 \end{align*}
 and formula \eqref{022803-26N} follows. 
\end{proof}

\medskip
 
{Recall that $\mathcal{L}_N^{(\alpha)}=\mathcal{S}_N^{(\alpha)}+\mathcal{A}_N^{(\alpha)}$, with $\mathcal{S}_N^{(\alpha)}$ and $\mathcal{A}_N^{(\alpha)}$ the symmetric and anti-symmetric part of $\mathcal{L}_N^{(\alpha)}$ respectively. For the anti-symmetric part, we have the following lemma:} 
\begin{lemma}\label{l.l2aN}
We have 
\begin{equation}
  \label{01303-26N}
  \begin{split}
    &
    \int_{\Om_N} \Q^{(N)}_V(t,u;
 f_N)   {\cal
    A}^{(\alpha)}_N\log \Q^{(N)}_V(t;
 f_N) (u)\nu_{S}^{(N)} (du)\\
 &
 =\int_{\Om_N}  {\cal
    A}^{(\alpha)}_N\Q^{(N)}_V(t;
 f_N)(u) \nu_{S}^{(N)} (du)=0.
 \end{split}
\end{equation}
\end{lemma} 

\begin{proof}
We have, cf \eqref{L-alN},
\begin{align*}
&{\cal
    A}^{(\alpha)} \log \Q^{(N)}_V(t;
 f_N) (u)=\frac{1}{6}\sum_{j\in\bT_N}\Big(N\nabla\alpha_2 \Big)\star
    G(u)\big(\frac{j}{N}\big)\partial_{u_j}\log \Q^{(N)}_V(t,u;
                 f_N) \\
  &
    =\frac{1}{6}\sum_{j\in\bT_N}\Big(N\nabla\alpha_2 \Big)\star
    G(u)\big(\frac{j}{N}\big) \frac{\partial_{u_j}  \Q^{(N)}_V(t,u;
                 f_N) }{\Q^{(N)}_V(t,u;
                 f_N)}.
\end{align*}
Therefore, the left hand side of \eqref{01303-26N} equals 
\[
  \begin{split}
        &
   \frac{1}{6}\sum_{j\in\bT_N}\int_{\Om_N}  \Big(N\nabla\alpha_2 \Big)\star
    G(u)\big(\frac{j}{N}\big)\partial_{u_j}  \Q^{(N)}_V(t,u;
                 f_N) \nu_{S}^{(N)} (du)\\
                 &
                  =\int_{\Om_N}  {\cal
    A}^{(\alpha)}_N\Q^{(N)}_V(t;
 f_N)(u) \nu_{S}^{(N)} (du)=0,
\end{split}
\]
by virtue of \eqref{012703-26bN}.
\end{proof}

\medskip

Putting together    \eqref{e.1281aN} and the
  results of Lemmas~\ref{l.l1aN}, and \ref{l.l2aN} we conclude the following.
\begin{proposition}
  We have 
\begin{equation}\label{e.1282aN}
  \begin{split}
    \frac{d {\mathbf{H}}_N^{(V)}(t;f_N) }{dt}=& \int_{\Om_N}  {\cal S}^{(\alpha)}_N   \Q^{(N)}_V(t,u;
 f_N)   \log \Q^{(N)}_V(t,u;
 f_N)  
 \nu_{S}^{(N)} (du) \\
 &
 +\EE_{{\mu_N(0)}} {\frak
  l}_{N,V}\Big(u_{N,V}(t) \Big).
\end{split}
\end{equation}
\end{proposition}
{One may view the first term on the r.h.s. as the dissipation and the second term as the forcing induced by the perturbation $V$.} Invoking the log Sobolev estimate, see Appendix \ref{AppA}, we obtain
 \begin{proposition}
   \label{prop023103-26N}
\begin{equation}\label{entropy-estN}
  \begin{split}
&    \frac{d {\mathbf{H}}_N^{(V)}(t;f_N) }{dt}\le  -{\frak g}_N
{\mathbf{H}}_N^{(V)}(t;f_N)  
 +\EE_{\mm{\mu_N(0)}} {\frak
  l}_{N,V}\Big(u_{N,V}(t) \Big),\\
 &
 \mbox{where}\quad {\frak g}_N :=\mm{8}N^2\sin^2\Big(\frac{\pi}{N}\Big).
\end{split}
\end{equation}
\end{proposition}
\proof From \eqref{e.1282aN} and \eqref{011104-26} we obtain
\begin{align*}
 & \frac{d {\mathbf{H}}_N^{(V)}(t;f_N) }{dt}= \frac{1}{2}
  \int_{\Om_N}  {N^3}\Delta{\mathbf \alpha_2}D_u\Q^{(N)}_V(t,u;
 f_N)\cdot D_u\log \Q^{(N)}_V(t,u;
  f_N)  \nu_{S}^{(N)} (du) \\
  &
    +\EE_{{\mu_N(0)}} {\frak
  l}_{N,V}\Big(u_{N,V}(t) \Big)
\end{align*}
\begin{align*}
 & = 
  \int_{\Om_N}  {2N^3}\Delta{\mathbf \alpha_2}D_u\big[\Q^{(N)}_V(t,u;
 f_N)\big]^{1/2}\cdot D_u\big[\Q^{(N)}_V(t,u;
 f_N)\big]^{1/2} \nu_{S}^{(N)} (du) \\
  &
    +\EE_{{\mu_N(0)}} {\frak
  l}_{N,V}\Big(u_{N,V}(t) \Big).
\end{align*}
Using Lemma \ref{lm013103-26aa} we get
\begin{align}
  \label{011304-26}
 & -\frac12
  \int_{\Om_N}  \big({N^3}\Delta\big){\mathbf \alpha_2}D_u\big[\Q^{(N)}_V(t,u;
 f_N)\big]^{1/2}\cdot D_u\big[\Q^{(N)}_V(t,u;
 f_N)\big]^{1/2} \nu_{S}^{(N)} (du) \notag\\
&\ge \inf_{|\xi|=1,\,\xi\perp 1}\big\la
  \big(-\frac12N^2\Delta\big)\xi,\xi\big\ra \int_{\Om_N}  \Q^{(N)}_V(t,u;
 f_N)\log \Q^{(N)}_V(t,u;
 f_N) \nu_{S}^{(N)} (du) .
\end{align}
We have
\begin{align*}
\inf_{|\xi|=1,\,\xi\perp 1}\big\la
  \big(-\frac12N^2\Delta\big)\xi,\xi\big\ra \ge  2N^2\sin^2(\frac{\pi}{N}),
\end{align*}
and \eqref{entropy-estN} follows.
\qed
%

\medskip

{Recall that $\nu_S^{(N)}$ is the Gaussian measure defined in \eqref{CNj} and ${\frak l}_{N,V}^2(u)$ was defined in \eqref{lN}.} Denote
\begin{equation}
  \label{hVN}
 {\frak v}_{N,V}:=\int_{\Om_N}  {\frak l}_{N,V}^2(u) \nu_{S}^{(N)} (du) =\frac{1}{N} \sum_{j\in\bT_N}  (N\nabla
V^{(\al_2)})\Big(\frac{j}{N}\Big) (N\nabla
 V )\Big(\frac{j}{N}\Big),
  \end{equation}
{which is the variance of ${\frak l}_{N,V}(u)$ under the unperturbed Gaussian measure. We have the following bound on the forcing term on the r.h.s. of \eqref{e.1282aN}:}
\begin{lemma}
  \label{lm012104-26}
  For any $\beta>0$ we have
  \begin{equation}\label{012104-26}
  \begin{split}
&     \EE_{{\mu_N(0)}} {\frak
  l}_{N,V}\Big({u_{N,V} (t)} \Big) 
\le  
  \frac{1}{\beta}\Big\{
{\mathbf{H}}_N^{(V)}(t;f_N)+ 
\frac{\beta^2}{2}   {\frak v}_{N,V} \Big\},\quad t\ge0,\,N=1,2,\ldots .
\end{split}
\end{equation}
  \end{lemma}
\proof
By the entropy inequality, see e.g. \cite[p. 338]{KL}, we get that
for any $\beta>0$  
 \begin{equation}\label{071704-26}
  \begin{split}
&    \EE_{{\mu_N(0)}} {\frak
  l}_{N,V}\Big({u_{N,V} (t)} \Big) 
=\int_{\Om_N}  \Q^{(N)}_V(t,u;
 f_N) {\frak l}_{N,V}(u)\nu_{S}^{(N)} (du) \\
&
\le \frac{1}{\beta}\Big\{
{\mathbf{H}}_N^{(V)}(t;f_N)+\log\Big(\int_{\Om_N}\exp\left\{ 
 \beta {\frak l}_{N,V}(u)\right\}\nu_{S}^{(N)} (du)\Big) \Big\} \\
&
=
  \frac{1}{\beta}\Big\{
{\mathbf{H}}_N^{(V)}(t;f_N)+ 
\frac{\beta^2}{2}  {\frak v}_{N,V}  \Big\} .
\end{split}
\end{equation}
The last equality follows from Gaussianity of $\nu_{S}^{(N)} $.
\qed

\medskip

Combining the above lemma with Proposition~\ref{prop023103-26N}, we
conclude that
\[
 \frac{d {\mathbf{H}}_N^{(V)}(t;f_N) }{dt}\le ( -{\frak g}_N+\beta^{-1})
{\mathbf{H}}_N^{(V)}(t;f_N) +\frac{\beta}{2}   {\frak v}_{N,V}.
\]
Applying Gronwall's inequality, we have
%
 \begin{theorem}
    \label{thm081704}
Suppose that  $\beta{\frak
    g}_N>1$. Then,  
    \begin{align}
  \label{031304-26fz}
    {\mathbf{H}}_N^{(V)}(t;f_N ) 
    &\le {\mathbf{H}}_N^{(V)}(0;f_N )\exp\left\{-({\frak
    g}_N\beta-1)\frac{t}{\beta}\right\} \notag
      \\
      &+  \frac{ {\frak v}_{N,V} \beta^{{2}} }{2 ({\frak
    g}_N\beta-1)}   \Big[1-\exp\left\{-({\frak
    g}_N\beta-1)\frac{t}{\beta}\right\}\Big]
    \end{align}
    for any  $  N=1,2,\ldots$ and $t\ge0$. 
\end{theorem}
  
  \subsection{The $L^p $ bound on the density}
\label{sec4.5}

{In this section, we derive an estimate on the $L^p-$norm of the density, using a similar strategy as the previous section, where we studied the relative entropy. The only difference is that, instead of the log Sobolev inequality, we resort to the spectral gap inequality associated with the underlying Gaussian measure.} 

  {For $p\ge 2$}, define
\begin{equation}
  \label{eq:7-1aNLp}
  \begin{split}
           \bar \Q_{p,N}(t ) &:= \int_{\Omega_N}  \Big[\Q^{(N)}_V(t,u;
 f_N)  \Big]^p
 \nu_{S}^{(N)} (du) \\
 &= \EE_{{\mu_N(0)}}    \Big[\Q^{(N)}_V(t, u_{N,V}(t) ;
 f_N)  \Big]^{p-1},\\
 \tilde \Q_{p,N}(t,u)&:=\big[ \Q^{(N)}_V(t,u;
 f_N)   \big]^{p}- \bar \Q_{p,N}(t) .
 \end{split}
\end{equation}
The  evolution of $ \bar \Q_{p,N}(t) $ satisfies
\begin{align*}
   & \frac{d \bar \Q_{p,N}(t) }{dt}= p\int_{\Om_N} \Big[\Q^{(N)}_V(t,u;
 f_N)  \Big]^{p-1}
    \frac{d}{dt} \Q^{(N)}_V(t,u;
 f_N)    
 \nu_{S}^{(N)} (du).
\end{align*}
Substituting, from the Fokker-Planck equation
$
\frac{d}{dt} \Q^{(N)}_V(t,u;
 f_N)=\Big( {\cal L}^{(\alpha)}_{N,V}\Big)^\star \Q^{(N)}_V(t,u;
 f_N)  ,
$
    we conclude
  \begin{equation}\label{e.1281aNL2}
    \begin{split}
       &
  \frac{d \bar \Q_{p,N}(t) }{dt} 
    =p\int_{\Om_N}  {\cal L}^{(\alpha)}_N   \big[\Q^{(N)}_V(t,u;
 f_N) \big]^{p-1}  \Q^{(N)}_V(t,u;
 f_N)   
     \nu_{S}^{(N)} (du) \\
   &
     +p\int_{\Om_N} \Q^{(N)}_V(t,u;
 f_N)   {\cal A}_{N,V}   \big[ \Q^{(N)}_V(t,u;
 f_N)   \big]^{p-1} 
   \nu_{S}^{(N)} (du).
\end{split}
  \end{equation}
For the second term on the right hand side of \eqref{e.1281aNL2}, we have
  \begin{equation}
    \label{022803-26NL2}
\begin{aligned}
   & \int_{\Om_N} \Q^{(N)}_V(t,u;
 f_N)   {\cal A}_{N,V}   \big[ \Q^{(N)}_V(t,u;
 f_N)   \big]^{p-1} 
   \nu_{S}^{(N)} (du)\\
 &
 = (p-1)  \int_{\Om_N} \big[ \Q^{(N)}_V(t,u;
 f_N)   \big]^{p-1} \Big[ \sum_{j\in\bT_N} (N \nabla V^{(\alpha_2)}) \big(\frac{j}{N}\big) \partial_{u_j}\Q^{(N)}_V(t,u;
 f_N)\Big] 
 \nu_{S}^{(N)} (du).
\end{aligned}
\end{equation}

\medskip

Recall that  ${\cal   L}^{(\alpha)}_N= {\cal   S}^{(\alpha)}_N+{\cal
  A}^{(\alpha)}_N$  is the generator of the unperturbed
 process
 $ 
   \Big(u_{N}\big(t,\frac{j}{N}\big)\Big)_{j\in\bT_N},
   $
that   satisfies \eqref{burgers.dSVN} with $V=0$.
  In consequence of \eqref{012703-26bN}
we have 
\begin{equation}
  \label{01303-26NL2}
  \begin{split}
    &
    \int_{\Om_N} {\cal
    A}^{(\alpha)}_N  \big[ \Q^{(N)}_V(t,u;
 f_N)   \big]^{p-1}  \Q^{(N)}_V(t;
 f_N) (u)\nu_{S}^{(N)} (du) \\
 &
{ =-\int_{\Om_N} {\cal
    A}^{(\alpha)}_N  \Q^{(N)}_V(t;
 f_N) (u) \big[ \Q^{(N)}_V(t,u;
 f_N)   \big]^{p-1} \nu_{S}^{(N)} (du)}   .
 \end{split}
\end{equation}
{Computing the left hand side  of \eqref{01303-26NL2}   using the form of $\mathcal{A}_N^{(\alpha)}$ in \eqref{L-alN}, the above term equals to 
\[
(p-1)\int_{\Om_N} {\cal
    A}^{(\alpha)}_N  \Q^{(N)}_V(t;
 f_N) (u) \big[ \Q^{(N)}_V(t,u;
 f_N)   \big]^{p-1} \nu_{S}^{(N)} (du).
\]}
Hence, the contribution from $\mathcal{A}_N^{(\alpha)}$ vanishes:
 \begin{equation}
  \label{01303-26NL20}
  \begin{split}
    &
    \int_{\Om_N} {\cal
    A}^{(\alpha)}_N  \big[ \Q^{(N)}_V(t,u;
 f_N)   \big]^{p-1}  \Q^{(N)}_V(t;
 f_N) (u)\nu_{S}^{(N)} (du) =0.
\end{split}
\end{equation}
\medskip

Define 
\begin{align*}
    &{\cal S}_{N,p}(t):=-p\int_{\Om_N}  {\cal S}^{(\alpha)}_N   \big[ \Q^{(N)}_V(t,u;
 f_N)   \big]^{p-1}   \Q^{(N)}_V(t,u;
 f_N)  
                 \nu_{S}^{(N)} (du) \\
  &
    =4(1-1/p)
  \int_{\Om_N}    (-\frac12{N^3}\Delta){\mathbf \alpha_2}D_u \tilde \Q_{p/2,{N}}(t,u) \cdot D_u \tilde \Q_{p/2,{N}}(t,u) \nu_{S}^{(N)} (du),
\end{align*}
where we used the expression of the Dirichlet form \eqref{011104-26} in the second step. 
Combining   \eqref{e.1281aNL2}, \eqref{022803-26NL2}
   together with \eqref{01303-26NL20}, we conclude the following identity
\begin{equation}\label{e.1282aNL2}
  \begin{split}
    \frac{d \bar \Q_{p,N}(t) }{dt}= &-{\cal S}_{N,p}(t)
+p(p-1) G_{N,p}(t)
\end{split}
\end{equation}
{with 
\begin{equation}\label{e.defFp}
\begin{aligned}
G_{N,p}(t):=\int_{\Om_N}& \big[ \Q^{(N)}_V(t,u;
 f_N)   \big]^{p-1}\\
 &\times
   \Big[ \sum_{j\in\bT_N} (N \nabla V^{(\alpha_2)}) \big(\frac{j}{N}\big) \partial_{u_j}\Q^{(N)}_V(t,u;
 f_N)\Big] 
 \nu_{S}^{(N)} (du).
 \end{aligned}
\end{equation}
As in \eqref{e.1282aN}, the first term on the right hand side of \eqref{e.1282aNL2} generates dissipation, while the second term represents the perturbative forcing. The next lemma deals with the forcing term through an integration by parts:} 
{\begin{lemma}\label{l.bdGNP}
We have 
\[
G_{N,p}(t)\leq p^{-1}\left(\frac{{\frak
      v}_{N,V,0}^{1/2} }{4{\frak g}_N^{1/2}(1-1/p) }  {\cal S}_{N,p}(t)+2 \bar \Q_{p/2,N}(t) \|\tilde \Q_{p/2,N}(t)\|_{L^2(\nu_{S}^{(N)} )}{\frak v}_{N,V}^{1/2}\right),
\]
with ${\frak
  v}_{N,V,0}$ given by
\begin{equation}
  \label{hVN0}
 {\frak v}_{N,V,0}:= \frac{1}{N} \sum_{j\in\bT_N}
      V^{(\al_2)}\Big(\frac{j}{N}\Big)
    V\Big(\frac{j}{N}\Big)   .
  \end{equation}
\end{lemma}}
\begin{proof}
Integrating by parts, we obtain, cf \eqref{lN},  
\begin{align*}
  G_{N,p}(t)
   &=-(p-1) \sum_{j\in\bT_N} \int_{\Om_N} \partial_{u_j}   \Q^{(N)}_V(t,u;
 f_N)   \big[ \Q^{(N)}_V(t,u;
   f_N)   \big]^{p-2} \\
  &
    \times\Big[  (N \nabla V^{(\alpha_2)})
    \big(\frac{j}{N}\big) \Q^{(N)}_V(t,u;
 f_N)\Big] 
    \nu_{S}^{(N)} (du) \\
  &
    + \int_{\Om_N} \big[ \Q^{(N)}_V(t,u;
 f_N)\Big]^p    {\frak l}_{N,V}(u)     
    \nu_{S}^{(N)} (du) .
\end{align*}
{The first term on the right hand side is precisely $-(p-1)G_{N,p}(t)$,} hence 
\begin{align}
  \label{031205-26}
   & p\, G_{N,p}(t)  
   = \sum_{j\in\bT_N} \int_{\Om_N} \big[ \Q^{(N)}_V(t,u;
 f_N)\Big]^p    {\frak l}_{N,V}(u)
    \nu_{S}^{(N)} (du)
    ={\rm I}_N+{\rm II}_N,\quad\mbox{where}
  \\
   &
  {\rm I}_N:=     \int_{\Om_N} \tilde \Q_{p/2,N}^2(t,u)   {\frak l}_{N,V}(u)
    \nu_{S}^{(N)} (du) ,\notag\\
  &
    {\rm II}_N :=2 \bar \Q_{p/2,N}(t) \int_{\Om_N} \tilde \Q_{p/2,N}(t,u) {\frak l}_{N,V}(u)
    \nu_{S}^{(N)} (du). \notag
\end{align} 
{To obtain the above decomposition, we also used the fact that $\int_{\Om_N}   {\frak l}_{N,V}(u)
    \nu_{S}^{(N)} (du)=0$. We deal with ${\rm I}_N$ and ${\rm II}_N$ separately.}

  {For ${\rm I}_N$, the idea is to perform a Gaussian integration by parts using the factor ${\frak l}_{N,V}(u)$. Indeed,}  
 applying Proposition \ref{prop021704-26} below, we estimate, cf.  \eqref{CNj}, 
\begin{align}
  \label{011704-26z}
  &
     {\rm I}_N\le \sqrt{2}   \Big[{4} \sum_{j\in\bT^N}
    (-N^2\Delta)^{-1}{\cal C}_N  ({-}\Delta) V\Big(\frac{j}{N}\Big)
    V\Big(\frac{j}{N}\Big)\Big]^{1/2}      
    \|\tilde \Q_{p/2,N}(t)\|_{L^2(\nu_{S}^{(N)} )}  
      \Big\{[4(1-1/p)]^{-1}{\cal S}_{N,p}(t)\Big\}^{1/2}\\
  &
    = \Big(\frac{1-1/p}{2}\Big)^{-1/2}{\frak v}_{N,V,0}^{1/2}\|\tilde \Q_{p/2,N}(t)\|_{L^2(\nu_{S}^{(N)} )}  {{\cal S}_{N,p}^{1/2}(t)}.\notag
\end{align}
Thanks to  the spectral gap estimate \eqref{021105-26} (see also
\eqref{entropy-estN}), we get
\begin{align}
  \label{011304-26L^2}
 &  
   \big[4(1-1/p)\big]^{-1}{\cal S}_{N,p} (t)   \ge 2{\frak g}_N \|\tilde \Q_{p/2,N}(t)\|^2_{L^2(\nu_{S}^{(N)} )} .
\end{align}
Hence,
\begin{equation}
  \label{041205-26}
 {\rm I}_N\le \frac{{\frak
      v}_{N,V,0}^{1/2} }{4{\frak g}_N^{1/2}(1-1/p) }  {\cal S}_{N,p}(t).
\end{equation}

Concerning ${\rm II}_N$ we use the Cauchy-Schwarz inequality and derive
(see \eqref{hVN})
\begin{equation}
  \label{021405-26}
 {\rm II}_N\le 2 \bar \Q_{p/2,N}(t) \|\tilde \Q_{p/2,N}(t)\|_{L^2(\nu_{S}^{(N)} )}{\frak v}_{N,V}^{1/2},
\end{equation}
which completes the proof.
\end{proof}
%

Applying Lemma~\ref{l.bdGNP} to \eqref{e.1282aNL2}, we have shown therefore the following.
\begin{proposition}
  \label{prop011305-26}
  Suppose that $p\in[2,+\infty)$. Then,   
\begin{equation}\label{011305-26}
  \begin{split}
    \frac{d \bar \Q_{p,N}(t) }{dt}&\le  -\delta(N,p){\cal S}_{N,p}(t)\\
    &+2(p-1) \bar \Q_{p/2,N}(t) \|\tilde \Q_{p/2,N}(t)\|_{L^2(\nu_{S}^{(N)} )}{\frak v}_{N,V}^{1/2}, 
    \end{split}
    \end{equation} where 
\[
\delta(N,p):=1-\frac{p}{{4}{\frak g}_N^{1/2} } {\frak
      v}_{N,V,0}^{1/2} .
\]
\end{proposition}

Using the above result and an induction argument, we conclude the following

\begin{theorem}\label{t.densityLp}
 {Fix any $p\geq1$. Suppose that $V\in H^1(\bT)$ and 
\begin{equation}
  \label{sNaNxz}
  \begin{split}
    & \|V\|_{L^2(\bT)}<\frac{1}{100p},\\
    &
    {\frak Q}(p):= \sup_{N\ge1}\bar \Q_{p,N}(0)<+\infty.
\end{split}\end{equation}
Then there exist $N_0>0$, $\lambda(p)>0$, and   ${\frak C}_{*,j}>0$, $j=1,2$
depending only on $p$, ${\frak Q}(p)$,  $\|V\|_{H^1(\bT)}$, such that 
\begin{equation}
  \label{051704-26z}
\bar \Q_{p,N}(t)\le {\frak C}_{*,1}\bar \Q_{p,N}(0)
\exp\left\{-\lambda(p)t\right\}+{\frak C}_{*,2}
,\quad t\ge0,\,N\ge N_0.
\end{equation}}
\end{theorem}
\proof 
    {The idea of the proof is straightforward: we make use of  \eqref{011305-26} and proceed by induction on $p=2^n$, $n=1,2,\ldots$. First, we apply the spectral gap inequality to the dissipative term on the right-hand side of \eqref{011305-26}. For the forcing term, we use Young’s inequality, so that the factor $\|\tilde \Q_{p/2,N}(t)\|_{L^2(\nu_S^{(N)})}$ can be absorbed by the dissipation, while the remaining factor is controlled by the induction hypothesis. The proof is then concluded by applying Gronwall’s inequality.}

  {First, note that by choosing $\|V\|_{L^2(\bT)}$ small, the constant ${\frak
      v}_{N,V,0}$ (that approximates $\|V\|_{L^2(\bT)}$ for large $N$)
    defined in \eqref{hVN0} can be also made small, and since $ {\frak g}_N\to 8\pi^2$ as $N\to\infty$, the dissipation rate $\delta(N,p)$ in \eqref{011305-26} can be made strictly positive. This is the role of the first condition in \eqref{sNaNxz}.  Throughout the proof,   the constant $C>0$ may change from line to line.}

In the
case $n=1$ we have
\begin{equation}\label{e.5319}
  \begin{split}
&\bar \Q_{1,N}(t)\equiv 1,\qquad\|\tilde \Q_{1,N}(t)\|_{L^2(\nu_{S}^{(N)} )}\le \bar
\Q_{2,N}^{1/2}(t), \quad\mbox{and}\\
&
\bar \Q_{2,N}(t)=\|\tilde \Q_{1,N}(t)\|^2_{L^2(\nu_{S}^{(N)} )}+1,\quad t\ge0.
\end{split}
\end{equation}
Using \eqref{011305-26}
for
$p=2$, spectral gap estimate \eqref{011304-26L^2} and then Young's inequality, we obtain
\begin{equation}\label{011305-26-1}
  \begin{split}
&    \frac{d \bar \Q_{2,N}(t) }{dt}=\frac{d   }{dt}\|\tilde
\Q_{1,N}(t)\|^2_{L^2(\nu_{S}^{(N)} )}\\
&
\le  -{C}\|\tilde
\Q_{1,N}(t)\|^2_{L^2(\nu_{S}^{(N)} )}+2  \|\tilde \Q_{1,N}(t)\|_{L^2(\nu_{S}^{(N)} )}{\frak v}_{N,V}^{1/2}
 \\
&
\le { -(C-\beta^{-1})\|\tilde
\Q_{1,N}(t)\|^2_{L^2(\nu_{S}^{(N)} )}+\beta{\frak
  v}_{N,V}},\quad \beta>0.
\end{split}
\end{equation}
Here $C$ is some universal constant. 
Choose $\beta=2/C$, plug in the last equation in \eqref{e.5319}, and apply the
Gronwall's inequality, we conclude that
\eqref{051704-26z} holds for $p=2$ with $\lambda(2)=C/2$.

{Suppose now that  we have shown \eqref{051704-26z} for $p=2^n$ for some
$n\ge1$.  For $p=2^{n+1}$, recall the differential inequality \eqref{011305-26}
\begin{equation}\label{e.5318}
  \begin{split}
    \frac{d \bar \Q_{p,N}(t) }{dt}&\le  -\delta(N,p){\cal S}_{N,p}(t)\\
    &+2(p-1) \bar \Q_{p/2,N}(t) \|\tilde \Q_{p/2,N}(t)\|_{L^2(\nu_{S}^{(N)} )}{\frak v}_{N,V}^{1/2}, 
    \end{split}
    \end{equation}
For the second term, by the induction condition we bound $2(p-1) \bar \Q_{p/2,N}(t) {\frak v}_{N,V}^{1/2} \leq C$ with some constant $C$ independent of $t$, and apply Young's inequality again,   the forcing term in \eqref{e.5318} is bounded as 
\[
2(p-1) \bar \Q_{p/2,N}(t) \|\tilde \Q_{p/2,N}(t)\|_{L^2(\nu_{S}^{(N)} )}{\frak v}_{N,V}^{1/2} \leq \beta^{-1}\|\tilde \Q_{p/2,N}(t)\|_{L^2(\nu_{S}^{(N)} )}^2+ C^2\beta.
\]
Applying spectral gap inequality  
\[
\delta(N,p){\cal S}_{N,p}(t)\geq 2\lambda(p)\|\tilde \Q_{p/2,N}(t)\|_{L^2(\nu_{S}^{(N)} )}^2,
\]
where $\lambda(p)>0$. Combining the above estimates and choosing $\beta$ large, we obtain
\[
\frac{d \bar \Q_{p,N}(t) }{dt}\leq - \lambda(p)\|\tilde \Q_{p/2,N}(t)\|_{L^2(\nu_{S}^{(N)} )}^2+ C.
\]
Further note that $\bar \Q_{p,N}(t)=\|\tilde \Q_{p/2,N}(t)\|_{L^2(\nu_{S}^{(N)} )}^2+|\bar \Q_{p/2,N}(t)|^2$, 
which implies 
\[
\frac{d \bar \Q_{p,N}(t) }{dt}\leq - \lambda(p)\bar \Q_{p,N}(t)+ C.
\]
Applying Gronwall's inequality again completes the proof for those $p=2^{n+1}$. The general cases of $p\geq1$ can be derived with an application of H\"older's inequality.}
\qed

\section{Funaki-Quastel Regularization}

\label{Reg}

{In this section, we study a type of specific regularization of the
  stochastic Burgers equation which was used in \cite{FQ}. The main
  goal is to pass to the limit of the discretization
  \eqref{burgers.dSVN} and derive a continuous stochastic PDE. The
  limiting stochastic PDE serves as a mollification of the singular stochastic Burgers equation. By applying the result from \cite{FQ}, we further prove the joint convergence of the solution and the driving noise as one removes the regularization.
}

\subsection{Solution of the   regularized
  stochastic Burgers equation}

Consider the $C^\infty$-smooth spatial regularization of the Wiener process
\begin{equation}
  \label{Weps}
W^{(\eps)}(t,x):=\eta^{(\eps)} \star W(t,x).
\end{equation}
where $\eta^{(\eps)}$ is the approximation of identity defined in \eqref{040405-26}. With $\eta_2^{(\eps)}=\eta^{(\eps)}\star \eta^{(\eps)}$, the covariance function of $W^{(\eps)}$ equals
\begin{equation}
  \label{weps}
  \begin{split}
&\EE \big[W^{(\eps)}(t,x)W^{(\eps)}(s,y)\big]=(t\wedge s)
R_\eps(x-y) ,\quad\mbox{where}\\
&
R_\eps(x-y)=\eta^{(\eps)} _2(x-y),\quad t,s\ge0,\quad  x,y\in \bT.
\end{split}
\end{equation}
Recall that $\xi(x)
    = \sum_{n=1}^{+\infty} \big(\xi^{(n,c)}{\rm
    c}_n(x)+\xi^{(n,s)}{\rm s}_n(x)\big)$ has the law of a white noise on $L^2(\bT)$ without zero mode, see \eqref{inv-meas1}. 
Let  $\nu_\eps$ be  the law of 
\begin{align}
  \label{inv-meas}
  &
    \xi_\eps(x)
    = \sum_{n=1}^{+\infty}\hat\eta^{(\eps)} (n)\big(\xi^{(n,c)}{\rm
    c}_n(x)+\xi^{(n,s)}{\rm s}_n(x)\big) .\notag
\end{align}
where $\xi^{(n,\iota)}$, $n=1,2,\ldots$, $\iota\in\{c,s\}$ are
i.i.d. standard (real valued) Gaussians. It is straightforward to check that $\xi_\eps=\eta^{(\eps)}\star \xi$.
The
measure $\nu_\eps$ is  supported on $C^\infty(\bT)$ for each $\eps>0$.

Recall that $V$ satisfies \eqref{V00}. 
We introduce, after \cite{FQ}, the regularized perturbed stochastic
Burgers equation (RSBE) of  the form
\begin{equation}
  \label{eq:SBEV}
  \begin{split}
&d u^{(\eps)}_V(t,x)=\Big[\tfrac12\Delta u^{(\eps)}_V
(t,x)+b^{(\eps)}[u_V^{(\eps)}](x) +\nabla V^{(\eta_2^{\eps})}(x)\Big]dt+\nabla dW^{(\eps)} (t,x),
\end{split}
\end{equation}
with the   notation, see \eqref{040405-26},
\begin{equation}
  \label{030904-26}
b^{(\eps)}[u_V^{(\eps)}](x):=\tfrac12\nabla\Big( |u_V^{(\eps)}|^2\star
\eta_2^{(\eps)}(t,x)\Big),\quad  V^{(\eta_2^{\eps})}(x):=V\star
\eta_2^{(\eps)} (x) .
\end{equation}
In the particular case $V=0$ we shall write $u^{(\eps)}(t)$ instead of
$u^{(\eps)}_0(t)$.
Suppose that the initial data $u\in L^2(\bT)$.
\begin{theorem}
  \label{thm010904-26}
  Suppose   $V $satisfies \eqref{V00}.  
  Then,   there exists a unique, up to
   indistinguishability,  $C(\bbT)$-valued
$\big({\cal W}_t\big)$-predictable  process $\big(u^{(\eps)}_V(t;u)\big)_{t\ge0}$ which  satisfies \eqref{eq:SBEV}
  in both the
  weak and mild sense, i.e. for any $\varphi\in C^2(\bT)$, 
  \begin{equation}
  \label{SBEEV1}
  \begin{split}
&\la u^{(\eps)}_V(t;u), \varphi\ra =\la u , \varphi\ra+ \int_0^t\Big[\tfrac12\la u^{(\eps)}_V (s;u), \Delta
\varphi\ra\\
& -\tfrac12\la   \big(u^{(\eps)}_V\big)^2 (\cdot;u)\star
\eta_2^{(\eps)}(s )  , \nabla\varphi\ra 
-\la V^{(\eta_2^{\eps})}  , \nabla
\varphi\ra\Big]ds
- \la W^{(\eps)}(t), \nabla\varphi\ra  
\end{split}
\end{equation}
and
\begin{equation}
  \label{SBEmildV}
  \begin{split}
&    u^{(\eps)}_V (t;u) = e^{1/2\Delta t}u   + \tfrac12\int_0^t
e^{1/2\Delta (t-s)}  \big((u^{(\eps)}_V)^2(s;u)      \star
 \nabla\eta_2^{(\eps)}\big)(s)   ds \\
 &+  \int_0^t
e^{1/2\Delta s} \nabla V^{(\eta_2^{\eps})}      ds 
+\int_0^t \nabla
e^{1/2\Delta (t-s)}d  W^{(\eps)} (s) 
\end{split}
\end{equation}
for $\nu_\eps$ a.s. $u$ and almost every realization of $\big(W^{(\eps)}
(t)\big)_{t\ge0}$.

In addition, when $V=0$, for each $t\geq0$, the law of each $u^{(\eps)}_V (t;u) $ in $C(\bT)$, under
$ \PP_{\nu_\eps}$ coincides with $
\nu_\eps$.
\end{theorem}
The proof of this result is fairly standard. We present it in Appendix \ref{appB}.

\subsection{Approximation of the regularized solution by
  discretization}

{Recall that  $Q_{N,V}$ is  the law of  $\big(u_{N,V}(t)\big)$,  the
solution of \eqref{burgers.dSVN}, with 
the initial data $u_{N,V}(0)$ distributed according to 
$\nu_{S}^{(N)} (du)$. Let $Q_\eps^{(V)}$ be the law of $(u_V^{(\eps)}(t))$, with $Q_\eps=Q_\eps^{(0)}$, and assume the initial data $u_V^{(\eps)}(0)$ is distributed according to $\nu_\eps$. Both $Q_{N,V}$ and $Q_\eps^{(V)}$ are supported on the space $\mathcal{C}=C([0,\infty),C(\bT))$. We have the following approximation as $N\to\infty$: } 
\begin{proposition}
  \label{prop011004-26}
  The sequence of laws $Q_{N,V}$ converge weakly over ${\cal C}$ to
  $Q_\eps^{(V)}$, as $N\to+\infty$.
  \end{proposition}
  \proof
  In the case of $V=0$, the convergence was proved in \cite{FQ}. The following argument deals with the case of $V\neq0$ through the Cameron-Martin change of measure.

  Suppose that $F: {\cal C}\to\bbR$ is bounded and continuous. {In \eqref{burgers.dSVN}, we view the perturbed equation as the unperturbed one with the driving force tilted as
  \[
  dw_j^{(\alpha)}(t)\mapsto dw_j^{(\alpha)}(t)+\tfrac{1}{\sqrt{N}}V^{(\alpha_2)}(\tfrac{j}{N})dt,
  \] so by Cameron-Martin theorem, }
  \begin{align}
    \label{021104-26}
  &  \EE_{\nu_S^{(N)}}F\big(u_{N,V}(\cdot)\big)=\EE_{\nu_S^{(N)}}\Big[F\big(u_{N}(\cdot)\big) \\
  &\times \exp\Big\{   
  \frac{1}{\sqrt{N}}\sum_{j\in\bT_N} w _j^{(\al)}(T) V\Big(\frac{j}{N}\Big)
  -\frac
{T  }{2N}\sum_{j\in\bT_N}  V^{{(\alpha_2)}}\Big(\frac{j}{N}\Big) V\Big(\frac{j}{N}\Big)
    \Big\} \Big].\notag
  \end{align}
  It follows from \eqref{V00} that there exists $\varphi\in C^3(\bT)$
  such that $V(x)=\nabla\varphi(x)$.
  By the fact that 
  $
V\Big(\frac{j}{N}\Big)=N\nabla \varphi\Big(\frac{j}{N}\Big)+O\Big(\frac{1}{N}\Big)
$, and recalling the definition of $\alpha_2$ in \eqref{e.defalpha2}, 
we can write
  \begin{align*}
    &\lim_{N\to+\infty}\EE_{\nu_S^{(N)}}F\big(u_{N,V}(\cdot)\big)\\
          &
      =\lim_{N\to+\infty}\EE_{\nu_S^{(N)}}\Big[F\big(u_{N}(\cdot)\big) \exp\Big\{   -
    \sqrt{N}\sum_{j\in\bT_N} \nabla w _j^{(\al)}(T) \varphi\Big(\frac{j}{N}\Big)
  -\frac
{T  }{2}\la R_\eps\star V,V\ra 
    \Big\} \Big],
  \end{align*}
  where $R_\eps$ is the spatial covariance function of the noise given by \eqref{weps}.
  From \eqref{burgers.dSVN} we can write $\sqrt{N}\sum_{j\in\bT_N} \nabla w _j^{(\al)}(T)
   \varphi\Big(\frac{j}{N}\Big)=I_1-I_2$ with 
  \begin{equation}
\label{041004-26}
\begin{split}
   I_1= \frac{1}{N}\sum_{j\in\bT_N}u_N
   \Big(t,\frac{j}{N}\Big) \varphi\Big(\frac{j}{N}\Big)-\frac{1}{N}\sum_{j\in\bT_N}u_N
   \Big(0,\frac{j}{N}\Big) \varphi\Big(\frac{j}{N}\Big)
   \end{split}
   \end{equation}
   and
   \begin{equation}
   \begin{aligned}
   I_2= \int_0^t\Big\{& \frac{1}{2N}\sum_{j\in\bT_N}u_N
   \Big(s,\frac{j}{N}\Big) N^2\Delta \varphi\Big(\frac{j}{N}\Big)
   ds\\
   &-  \frac{1}{6N}\sum_{j\in\bT_N} \alpha_2\star
G\Big (u_N(s))\Big)(j)
N\nabla\varphi\Big(\frac{j}{N}\Big)\Big\}ds.
\end{aligned} 
\end{equation}
In addition,  we have the uniform integrability that can be seen from
\begin{align*}
  &
      \EE_{\nu_S^{(N)}}\Bigg\{\Big[F\big(u_{N}(\cdot)\big) \exp\Big\{   -
    \sqrt{N}\sum_{j\in\bT_N} \nabla w _j^{(\al)}(T) \varphi\Big(\frac{j}{N}\Big)
  -\frac
{T  }{2}\la  R_\eps\star V,V\ra_{L^2(\bT)}
    \Big\} \Big]^2\Bigg\}\\
  &
     \le \|F\|_\infty^2 \EE_{\nu_S^{(N)}}\Bigg\{ \exp\Big\{   
    2\sqrt{N}\sum_{j\in\bT_N} w _j^{(\al)}(T) \nabla\varphi\Big(\frac{j}{N}\Big)\Big\}
    \Bigg\} \\
  &
   \le \|F\|_\infty^2 \EE_{\nu_S^{(N)}}\Bigg\{ \exp\Big\{   
     \frac{2T\|\varphi'\|_\infty^2 }{N}\sum_{j,j'\in\bT_N}
    \al_2 ( j-j')  \Big\}
     \Bigg\}\le C
  \end{align*}
  for $N=1,2,\ldots$, where $C$ is some constant independent of $N$.
  
  Since the laws of $u_{N}(\cdot)$  converge weakly to $u^{(\eps)}$, see
  \cite[Section 2.5]{FQ}, we
 conclude that 
 \begin{align*}
  & \lim_{N\to+\infty} \EE_{\nu_S^{(N)}}F\big(u_{N,V}(\cdot)\big)\\
  &=\EE_{\nu_\eps}\Big[F\big(u^{(\eps)}(\cdot)\big) \exp\Big\{   -
    \int_{\bbT}\nabla W^{(\eps)}(T,x) \varphi(x)dx
  -\frac
{T  }{2}\la R_\eps\star V,V\ra 
    \Big\} \Big]\\
   &
     =\EE_{\nu_\eps}\Big[F\big(u^{(\eps)}(\cdot)\big) \exp\Big\{   
    \la W^{(\eps)}(T), V\ra
  -\frac
{T  }{2}\la  R_\eps\star V,V\ra     \Big\} \Big]{=\EE_{\nu_\eps}\Big[F(u_V^{(\eps)}(\cdot))\Big]},
  \end{align*}
  and the conclusion of the proposition follows.
  \qed

  \subsection{Approximation of the solution of the SBE by the solutions of the regularized SBE}

{In this section, we remove the regularization and study the convergence as $\eps\to0$. When $V=0$, the convergence of $u^{(\eps)}$ to the solution of the stochastic Burgers equation was proved in \cite{FQ}. To treat the case $V\neq0$, we again make use of the Cameron-Martin change of measure, under which the noise  shows up in the Radon-Nikodym derivative. Thus, we need the joint convergence of the solution and the noise, and this section is devoted to the proof of such a result.}

  Recall that $\big(W(t)\big)$ is the cylindrical
Wiener process on $L^2(\bT)$ that was used in the construction of $W^{(\eps)}$ in \eqref{Weps}. 
      Fix $T>0$. {Assume $u^{(\eps)}$ solves \eqref{eq:SBEV} with $V=0$, with $u^{(\eps)}(0)$ distributed according to $\nu_\eps$. } The laws of the random elements
  \[
  {\frak X}_\eps(T):=\Big(u^{(\eps)}(T) , u^{(\eps)}(0),W (\cdot)\Big)
  \] are
  tight in the space ${\cal
    H}:=\rH^2\times C\big([0,T];\rH\big)$. We have 
      \begin{proposition}
        \label{prop012404-26}
        The laws of ${\frak X}_{\eps}(T)$ weakly converge
  as $\eps\to0$. The  limiting law
  coincides with that of  ${\frak X}:=\Big(u(T),u(0),W(\cdot)\Big)$ where
  $u$ solves  \eqref{011804-26} with $V=0$ driven by $W$ and $u(0)$ is distributed according to $\nu$.
\end{proposition}
\proof
  Recall that $u^{(\eps)}(t)$ solves \eqref{eq:SBEV}, with $V=0$, in the sense of Theorem
 \ref{thm010904-26}, starting at stationarity. Then, since the  equation
   is with an additive noise, equipped with a
 cylindrical Wiener process $W(\cdot)$ and $u^{(\eps)}(0)$, one can
 construct a deterministic, Borel measurable mapping ${\cal S}: \rH \times C\big([0,T];\rH\big)
 \to C\big([0,T];\rH\big)$ such that
 \begin{equation}
   \label{012404-26}
u^{(\eps)}(\cdot)={\cal S}\Big(u^{(\eps)}(0),W\Big).
\end{equation}
Suppose that $\eps_n\to0$ and  
  the laws of ${\frak
  X}_{\eps_n}(T):=\Big(u^{(\eps_n)}(T),
u^{(\eps_n)}(0),W (\cdot)\Big)$ weakly converge to the law of some ${\frak
  X} $. According to the Skorokhod representation theorem, there exists
a probability space $(\tilde \Omega,\tilde{\cal F},\tilde \PP)$
and the random elements  
\[
\hat{\frak
  X}_{\eps_n}(T):=\Big(\hat u^{(\eps_n)}(T),\hat u^{(\eps_n)}(0),\hat W_{\eps_n} (\cdot)\Big),
\] $\hat {\frak
  X} (T)$  such that for each $n$
\begin{equation}
  \label{011505-26a}
{\rm Law}\Big(\hat{\frak
  X}_{\eps_n}(T) \Big)={\rm Law}\Big( {\frak
  X}_{\eps_n}(T) \Big)\quad \mbox{and}\quad {\rm Law}\Big(\hat{\frak
  X} (T) \Big)={\rm Law}\Big( {\frak
  X} (T) \Big)
\end{equation}
and
\begin{equation}
  \label{022404-26}
\lim_{n\to+\infty}\hat{\frak
  X}_{\eps_n}(T) =\hat{\frak
  X} (T) ,\quad \tilde \PP-\mbox{a.s.}
\end{equation}
The above convergence holds in the metric of the space $\rH^2 \times
C\big([0,T];\rH\big)$.
In particular, from \eqref{011505-26a} we conclude that 
\begin{align}
  \label{021505-26}
  \hat u^{(\eps_n)}(\cdot) = {\cal S}\Big(\hat u^{(\eps_n)}(0),\hat
  W_{\eps_n}\Big)  .
\end{align}

Let
$\big(\tilde u^{(\eps_n)}(t)\big)_{t\in[0,T]}$ be the solution of
\eqref{eq:SBEV}, with $V=0$,  with the driving force  $\hat W_{\eps_n}
(\cdot)$ and the initial data $\hat u^{(\eps_n)}(0)$. According to
\eqref{012404-26} and \eqref{021505-26}
we have
$$
\tilde u^{(\eps_n)}(\cdot)={\cal S}\Big(\hat u^{(\eps_n)}(0),\hat W_{\eps_n}\Big) = \hat u^{(\eps_n)}(\cdot).
$$
Therefore,
\[
  \begin{split}
& {\rm Law}\Big(\tilde u^{(\eps_n)}(T),
\hat u^{(\eps_n)}(0),\hat W_{\eps_n} (\cdot)\Big) ={\rm Law}\Big( u^{(\eps_n)}(T), 
u^{(\eps_n)}(0), W_{\eps_n} (\cdot)\Big)   
\end{split}
\]
and  the random elements
$ \Big(\tilde u^{(\eps_n)}(T),\hat W_{\eps_n} (\cdot),
\hat u^{(\eps_n)}(0)\Big)
$
converge $\tilde\PP$ a.s.
It has been shown in \cite[Section 3.5]{FQ} that the solutions
$\big(\tilde  u^{(\eps_n)}(t)\big)_{t\in[0,T]}$
converge weakly (in the sense of the Hilbert space topology) over
$L^2\big([0,T]\times \bbT\times \tilde \Omega)$  to $\big(   u
(t)\big)_{t\in[0,T]}$ that satisfies \eqref{011804-26} {with $V=0$} and starts at
stationarity. The conclusion of the proposition then follows.
\qed

\section{The proof of Theorem~\ref{cor022404-26}}
\label{s.last}

{Recall the goal was to derive quantitative properties on the invariant measure $\nu_V$ for the perturbed dynamics 
\begin{equation}\label{e.uv61}
d u_V(t,x)=\Big(\frac12\Delta u_V(t,x)+\frac12 \nabla u_V^2(t,x)+\nabla V(t,x)\Big)dt +d\nabla W(t,x).
\end{equation}
The existence and uniqueness of $\nu_V$ are guaranteed by Theorem~\ref{thm012504-26}. Through the discussion in the previous sections, we constructed a family of discrete and continuous approximations, $\{u_{N,V}\}_N$ and $\{u^{(\eps)}\}_\eps$ which solve \eqref{burgers.dSVN} and \eqref{eq:SBEV} respectively. The properties of $u_V$ can be accessed from $u_{N,V}$, by first sending $N\to\infty$ then sending $\eps\to0$. Most importantly, we have obtained estimates on the law of $u_{N,V}$, which are uniform in $N$ and $\eps$, see Theorems~\ref{thm081704} and \ref{t.densityLp}.  The goal of this section is to transfer those properties of $u_{N,V}$ to $u_V$, and thereby prove Theorem~\ref{cor022404-26}.}

\label{L1s}

\subsection{The proof of Proposition \ref{sec1aC}}

\label{secL1sa}

{Recall that $\mu_{t,V}(F)$ is the law of $u_V(t;F)$, the solution to \eqref{e.uv61} with $u_V(0)$ distributed according to $F(\cdot)\in D(\nu)$.  The goal here is to show that, for any $t>0$, $\mu_{t,V}(F)$ is absolutely continuous with respect to $\nu$. Since $F$ is fixed in this section, we abbreviate it as $\mu_{t,V}$.}

 Suppose that $A\subset \rH$ is such that $\nu(A)=0$. Then
\begin{align*}
 & \mu_{t,V} (A)=\int_{\rH }\nu(du)F(u)\EE
   1_A\Big(u_V(t;u)\Big) \\
  &
    =\int_{\rH }\nu(du)F(u)\EE \Big[
   1_A\big(u(t;u)\big)
    \exp\Big\{\int_0^t\la dW(s),V\ra-\frac
t2
\|V\|_{L^2(\bT)}^2\Big\}\Big].
\end{align*}
Using the H\"older inequality we can write for any $1/p+1/q=1$ and $p\in(1,+\infty)$
\begin{align}
  \label{031605-26}
 & \mu_{t,V} (A)\le \left\{\int_{\rH }\nu(du)F(u)\EE
   1_A\Big(u(t;u)\Big)\right\}^{1/q}\notag\\
  &
    \times\left\{\int_{\rH }\nu(du)F(u)
    \EE\Big[\exp\Big\{p \int_0^t\la dW(s),V\ra-\frac
{pt}2
\|V\|_{L^2(\bT)}^2\Big\}\Big]\right\}^{1/p}
\\
  &
    = \left\{\int_{\rH }\nu(du)F(u)\EE
   1_A\Big(u(t;u)\Big)\right\}^{1/q}    \exp\Big\{ \frac
{(p-1)t}2
\|V\|_{L^2(\bT)}^2\Big\} . \notag
\end{align}
Since $F\in D(\nu)$, we can write by the Lebesgue dominant convergence
theorem and the invariance of $\nu$ under the dynamics of $u(t;u)$
\begin{align*}
  &\int_{\rH }\nu(du)F(u)\EE
  1_A\Big(u(t;u)\Big)\\
  &
    =\lim_{M\to+\infty}\int_{\rH }\nu(du)\big(M\wedge F(u)\big)\EE
   1_A\Big(u(t;u)\Big)\le \lim_{M\to+\infty}M\nu(A)=0.
\end{align*}

Thus, the absolute continuity of $\mu_{t,V} $ w.r.t. $\nu$  follows. As a
result, its Radon-Nikodym derivative ${\cal
  P}_t^VF$ satisfies
\eqref{f-p}.  The fact that the family $\big({\cal
  P}_t^VF\big)$ forms a semigroup is the consequence  of the semigroup
property of the adjoint semigroup to $\big({\cal U}_t^V\big)$. We show
its strong continuity.

For this purpose, for a fixed $t$ consider the time reversed stationary
solution
of the SBE with $V=0$
\begin{equation}
  \label{tu0}
  \tilde u (s):=u(t-s),\quad s\in[0,t].
\end{equation}
According to \cite[Proposition 1.1]{BQS11},
  the law $\tilde Q$ of $\tilde u$ on ${\cal C}_T:=C\big([0,T];\rm H\big)$ coincides with that of
  the solution of
  \begin{equation}
  \label{tSBEE1}
  \begin{split}
&d \tilde u (t,x)=\big(\tfrac12\Delta\tilde
u(t,x)-\tfrac12\nabla  \tilde u^2   \big)dt+ d \nabla\tilde
W   (t,x)
,\quad\mbox{with}\\
&
\mbox{\rm law}\big(\tilde u (0)\big)=\nu, 
\end{split}
\end{equation}
where $\big(\tilde W 
(t)\big)$ is some    $L^2(\bT)$-cylindrical Wiener process, adapted w.r.t. the natural filtration of
$\big(\tilde u (s)\big)_{s\in[0,t]}$. We understand the solution of
the above equation analogously, as   in the case of $u(\cdot)$, see
Section \ref{sec1.1}. Using the above interpretation we can also
define 
 $\tilde u(\cdot ;u)$ - the solution of \eqref{tSBEE1}
satisfying  $\tilde u(0 ;u)=u$ for $\nu$ a.s. $u\in\rH$.

By \cite[Theorem 2.4, see also Remark 2.7 and Section 2.3]{GP}
  we have  
\begin{equation}
  \label{021904-26}
   \begin{split}
   &- \int_0^t
\la \tilde u(s),\Delta \varphi\ra ds- \la \varphi,\nabla\tilde
W (t)\ra  
=
  \la \varphi,\nabla 
W (t)\ra, \quad \varphi\in C^\infty(\bbT) .
\end{split}
\end{equation}
Let
 \begin{equation}
  \label{041605-26}
           \widehat{\cal X}_t (V,  u):=\int_0^t \la \tilde u (s;u), 
 \nabla   V\ra ds
                 - \la   \tilde W (t),V\ra  -\frac
                 {t }{2}\|V\|_{L^2(\bT)}^2 .
                 \end{equation}
 The strong continuity of $\big({\cal
  P}_t^VF\big)$  is then a consequence  of the following result.
\begin{lemma}
  \label{thm012104-26}
  Assume \eqref{V00} and $t>0$. Then, for any $q\in(1,+\infty)$
we have
  \begin{equation}
  \label{051605-26}
\int_{\rH }\exp\Big\{q\widetilde{\cal X}_t(V, u
      )
   \Big\} \nu(du)=\exp\left\{\frac
{tq(q-1)}2
\|V\|_{L^2(\bT)}^2\right\}<+\infty.
\end{equation}
In addition, $F\in L^p(\nu)$, where
  $p\in(1,+\infty)$, we have 
\begin{equation}\label{021804-26}
\begin{aligned}
&{\cal P}_t^VF(u)  =
   \EE \Big[
  F\big(\tilde u(t ;u)\big) \exp\Big\{\widetilde{\cal X}_t(V, u
      )
   \Big\}\Big],\quad \nu \mbox{ a.s. in $u$.}
\end{aligned}
\end{equation}
\end{lemma}
\proof
 Suppose first that $F,G\in L^\infty(\nu)$. Using the Cameron-Martin-Girsanov change of
 measure formula and the time reversal in \eqref{tu0}, we can write
 \[
 \begin{aligned}
  \int_{\rH }G(u)&{\cal P}_t^VF(u)\nu(d u)  =
   \int_{\rH} \EE \Big[ G(u(t))F(u(0))
   \exp\Big\{  \ \la W(t),V\ra-\frac
{t}2
\|V\|_{L^2(\bT)}^2\Big\} \Big]  \\
&=\int_{\rH} \EE \Big[ G(\tilde{u}(0;u))F(\tilde{u}(t;u))
   \exp\Big\{  \ \la W(t),V\ra-\frac
{t}2
\|V\|_{L^2(\bT)}^2\Big\} \Big].
       \end{aligned}
       \]
Thanks to \eqref{V00}
we can find  $\varphi$ such that $\varphi' = V$ in \eqref{021904-26}, so we have 
\begin{align}
\label{021605-26} \int_0^t
\la \tilde u(s),\nabla V\ra ds- \la  \tilde
W (t),V\ra  
=
  \la  
  W (t),V\ra .
\end{align}
Thus we conclude \eqref{021804-26}. 
From \eqref{021605-26} it follows also that
{\begin{align}
   \label{011605-26}
 &\int_{\rH }\exp\Big\{q\widetilde{\cal X}_t(V, u
      )
   \Big\} \nu(du) =
   \int_{\rH} \EE \Big[ 
   \exp\Big\{  \ q\la W(t),V\ra-\frac
{qt}2
\|V\|_{L^2(\bT)}^2\Big\} \Big], 
       \end{align}}
which leads to \eqref{051605-26}.
\qed


As a direct corollary of Lemma \ref{thm012104-26} we conclude the following.
\begin{corollary}
  \label{sec1aC1}
  For   $p\in(1,+\infty)$ and $1/p+1/q=1$ we have 
\begin{equation}
  \label{022808-25a}
  \|{\cal P}_t^VF\|_{L^{p}(\nu)}\le \| F\|_{L^{p}(\nu)} \exp\Big\{ \frac
{(q-1)t}2
   \|V\|_{L^2(\bT)}^2\Big\} ,\quad F\in L^p(\nu).
 \end{equation}
\end{corollary}

\subsection{The proof of Theorem \ref{cor022404-26}}

\label{secL1sb}

Now we are ready to prove the main result.  Throughout the section, denote 
\[
f_N^{(V)}(t,u):=\Q^{(N)}_V(t,u;
 f_N)
 \]  the density of $u_{N,V}(t)$ with respect to $\nu_S^{(N)}$, with $u_{N,V}(0)$ distributed according to $f_N(u)\nu_S^{(N)}(du)$.

\subsubsection{Existence of an invariant density and bounds on the relative entropy}

\label{L2s}

The first step is to send $N\to\infty$ in the entropy bound obtained in Theorem~\ref{thm081704}. 

Recall that $
\big(u_{N,V}(t)\big)$ is  the discretization constructed in
Section \ref{discr}, and $u^{(\eps)}_V$ is the solution to  \eqref{eq:SBEV}. 
Assume that $G\in D(\nu_\eps)$ is the density of $\mu^{(\eps)}_V$ -
the law of $u^{(\eps)}_V(0)$ -
w.r.t. $\nu_\eps$. 
Suppose furthermore that  both $G$ and $1/G$ belong
to  $C_b\big(C(\bT)\big)$. Let
\begin{equation}
  \label{iG}
  \iota(G):=\sup_{u\in C(\bT)}\Big\{G(u),\frac{1}{G(u)}\Big\}.
\end{equation}
We have embedded  $\Om_N\subset C(\bT)$, so the restriction of $G_{|\Om_N}$ belongs to
$C_b(\Om_N)$.

{We choose the initial distribution $u_{N,V}(0)$ to be $ f_N(u)\nu_S^{(N)}(du)$ with
$
f_N(u)= G_{|_{\Omega_N}}/Z_N
$, where $Z_N$ is the normalization factor.}
%
 {It is clear that $f_N\leq \iota(G)^2$ and $Z_N\to1$, as $N\to\infty$.} 



The relative entropy of the law of $u_{N,V}(t)$ with respect to $\nu_S^{(N)}$  is 
 \[
 {\mathbf{H}}_N^{(V)}(t)=\int_{\Omega_N} f_N^{(V)}(t,u)\log f_N^{(V)}(t,u)\nu_S^{(N)}(du).
 \]
 By the fact of $f_N\leq \iota(G)^2$, we have the initial entropy is uniformly bounded: 
\begin{equation}
  \label{muN01}
{{\mathbf{H}}_N^{(V)}(0)  \le 2\log \iota(G)},\quad N=1,2,\ldots.
\end{equation}

According to \cite[Lemma 1.4.3 (p. 29)]{DE}  we have
\[\begin{aligned}
  \label{031304-26}
  {\mathbf{H}}_N^{(V)}(t)=\sup_{F\in C_b(\Om_N)}\Big\{&\int_{\Om_N}F(u) f_N^{(V)}(t,u)\nu_S^{(N)}(du)\\
 &-\log\Big(\int_{\Om_N}e^{F(u)}\nu_S^{(N)}(du)\Big)\Big\}.
\end{aligned}\]

Using \eqref{muN01} together with \eqref{031304-26fz} (see
\eqref{entropy-estN}), we conclude
that for any $F\in C_b(\Om_N)$  and $\beta{\frak g}_N>1$    \begin{align}
  \label{031304-26ff}
    &\EE_{\mu_N} \Big[ F\big(u_{N,V}(t)\big)
      \Big]-\log\Big(\int_{C(\bT)}e^{F(u)}\nu_S^{(N)}(du)\Big)\le {2\log \iota(G)}\exp\left\{-({\frak
    g}_N\beta-1)\frac{t}{\beta}\right\}  
      \\
      &+  \frac{ \beta^2 }{2 ({\frak
    g}_N\beta-1)}   \Big[1-\exp\left\{-({\frak
    g}_N\beta-1)\frac{t}{\beta}\right\}\Big] \Big[\frac{1}{N} \sum_{j\in\bT_N}  (N\nabla
V^{(\al_2)})\Big(\frac{j}{N}\Big) (N\nabla
 V )\Big(\frac{j}{N}\Big)\Big]\notag
    \end{align}
    for any  $  N=1,2,\ldots$ and $t\ge0$, with $\lim_{N\to+\infty} {\frak
  g}_N=8\pi^2$. 
    
   By virtue of Proposition \ref{prop011004-26}, letting   $N\to+\infty$, we
conclude  that for any $\beta>1/(8\pi^2)$ and    $F\in C_b\big(C(\bT)\big)$
 \begin{align}
  \label{031304-263f}
    &\EE\Big[ F\big(u_V^{(\eps)} (t)\big) G\big(u_V^{(\eps)} (0)\big)
      \Big]-\log\Big(\int_{C(\bT)}e^{F(u)}\nu_\eps(du)\Big)\le \mm{2}\log \iota(G)\exp\left\{-(8\pi^2\beta-1)\frac{t}{\beta}\right\}  
      \\
      &+  \frac{ \beta^2}{ 16\pi^2\beta-2}
        \Big[1-\exp\left\{-(8\pi^2\beta-1)\frac{t}{\beta}\right\}
        \Big] \langle R_\eps \nabla V,
   \nabla  V\rangle .  \notag
    \end{align}

The next step is to pass to the limit of $\eps\to0$ in \eqref{031304-263f}. 
    Suppose now that $G\in D(\nu)\cap C_b(\rH)$ satisfying
    $
  \iota(G)=\sup_{{u\in \rH}}\Big\{G(u),\frac{1}{G(u)}\Big\}<+\infty.
$
Then  $G_\eps=\frac{G}{\int_{C(\bT)}Gd\nu_\eps}$ is bounded from above and below by some positive constants,  and 
$\lim_{\eps\to0+} \|G_\eps-G\|_{\infty}=0$.
Using the Cameron-Martin-Girsanov change of
 measure formula,  
       Proposition \ref{prop012404-26}, {and a straightforward check of uniform integrability, we conclude for any $F\in C_b(\rH)$ that } 
      \begin{equation}
        \label{041805-26}
        \begin{split}
          &\lim_{\eps\to0}\EE_{\mu_V }\Big[ F\big(u_V^{(\eps)} (t)\big) G_\eps\big(u_V^{(\eps)} (0)\big)
          \Big]\\
          &=\lim_{\eps\to0}\EE \Big[ F\big(u^{(\eps)} (t)\big) G_\eps\big(u^{(\eps)} (0)\big)
        \exp\Big\{  \la \eta^{(\eps)}\star W(t),V\ra-\frac
{t}2
{\la R_\eps\star V, V\ra}\Big\} \Big]\\
&= \EE \Big[ F\big(u (t)\big) G\big(u  (0)\big)
        \exp\Big\{  \la   W(t),V\ra-\frac
{t}2
\|V\|_{L^2(\bT)}^2\Big\} \Big]=\int_{\rH}({\cal P}_t^VG) Fd\nu.
\end{split}
        \end{equation}
        As a result, using \eqref{031304-263f} with $G$ replaced by $G_\eps$ and then
        letting $\eps\to0$, we get
                 \begin{align}
  \label{011805-26}
   & \int_{\rH}({\cal P}_t^VG) Fd\nu-\log\Big(\int_{\rH}e^{F(u)}\nu(du)\Big)\le 
     2\log \iota(G)\exp\left\{-(8\pi^2\beta-1)\frac{t}{\beta}\right\}  \notag
      \\
      &+   \frac{ \beta^2}{ 16\pi^2\beta-2}
        \Big[1-\exp\left\{-(8\pi^2\beta-1)\frac{t}{\beta}\right\}
        \Big] {\la \nabla V,\nabla V\ra}.  
        \end{align}
        Taking the supremum over $F\in C_b(\rH)$ we conclude that the
        relative entropy of the law of $u_V(t)$ with respect to the white noise measure $\nu$ satisfies
\begin{align}
  \label{011805-26}
   & {\mathbf{H}}_{V}(t;G )=\int_{\rH}{\cal P}_{t}^VG\log {\cal P}_t^VG \,d\nu\le 2 \log \iota(G)\exp\left\{-(8\pi^2\beta-1)\frac{t}{\beta}\right\}  \notag
      \\
      &+   \frac{ \beta^2}{ 16\pi^2\beta-2}
        \Big[1-\exp\left\{-(8\pi^2\beta-1)\frac{t}{\beta}\right\}
        \Big] \la \nabla V,\nabla V\ra .  
    \end{align}

  Given $C>0$, define
  $$
{\cal D}(C):=\big[G\in D(\nu): {\mathbf{H}}(G )\le C  \big],
$$
where
$
{\mathbf{H}} (G)=\int_{\rH} G\log  G d\nu.
$
The set is convex and closed in the strong topology in $L^1(\nu)$.
By the Dunford-Pettis
criterion, it is also 
weakly compact in $L^1(\nu)$.
  Suppose that
  \begin{equation}
    \label{c-s}
    {{\frak C}_V(\beta):=\frac{ \beta^2}{ 16\pi^2\beta-2}\la \nabla V,\nabla V\ra}
        .
  \end{equation}
 Using the density argument, {the contraction property of ${\cal P}_t^V$,} together with \eqref{011805-26}, we
 conclude that for any $F\in D(\nu)$, 
   \begin{equation}
    \label{062404-26}
    \lim_{t\to+\infty}{\rm dist}\Big({\cal P}_t^VF, {\cal
      D}\big({\frak C}_V(\beta)\big)\Big)=0,
  \end{equation}
  where ${\rm dist}(F,S)$ is the infimum of $\|F-G\|_{L^1(\nu)}$ over $G\in S$.

  \medskip

From \eqref{062404-26}
and \cite[Theorem 1.1]{komornik} we conclude that there exist
$G_1,\ldots,G_d\in D(\nu)$, that are disjointly supported,  and $H_1,\ldots,H_d\in L^\infty(\nu)$ such that
\begin{itemize}
\item[i)] ${\cal P}_1^VG_i=G_{\alpha(i)}$, where
  $\alpha$ is a permutation of $\{1,\ldots,d\}$,
\item[ii)] for any $F\in L^1(\nu)$ we have
   \begin{equation}
    \label{nn072404-26}
    \lim_{n\to+\infty} \|{\cal P}_n^VF-\sum_{j=1}^dG_j\int_{\rH}H_jFd\nu\|_{L^1(\nu)}=0 .
  \end{equation}
  \end{itemize}
Asymptotic stability, as stated in Theorem \ref{thm012504-26}, implies
that $d=1$. Hence, we have shown the existence of  density  $\bar F_V$
as in \eqref{012704-26} that satisfies also \eqref{rel-ent} and \eqref{n072404-26}.

\subsubsection{The proof of the $L^p$ integrability}

 Using \eqref{051704-26z} with $f_N^{(V)}(0)\equiv 1$ (starting from $\nu_S^{(N)}$), we conclude that 
  for any $p\ge1$, if $\|V\|_{L^2(\bT)}\leq \tfrac{1}{100p}$, there exist $N_0>0$ and   ${\frak C}_{*,j} $, $j=1,2$
depending only $p$, $\|V\|_{H^1(\bT)}$ (in this case ${\frak Q}(p)\equiv 1$), such that 
\begin{equation}
  \label{051704-26z2}
\int_{\Om_N}\big(f_N^{(V)}(t,u)\big)^p\nu_S^{(N)}(du)\le {\frak C}_{*,1} 
\exp\left\{-\lambda(p)t\right\}+{\frak C}_{*,2} 
,\quad t\ge0,\,N\ge N_0.
\end{equation}
Since
\[
\begin{aligned}
  \label{061805-26}
  &\frac{1}{p}\int_{\Om_N}\big(f_N^{(V)}(t,u)\big)^p\nu_S^{(N)}(du)\\
  &=\sup_{F\in C_b(\Om_N)}\Big\{\EE_{\mm{\nu_S^{(N)}}} \Big[ F\big(u_{N,V}(t)\big)
      \Big]-\frac1q \int_{\Omega_N}|F(u)|^q\nu_S^{(N)}(du)\Big\},
\end{aligned}
\]
letting first $N\to+\infty$ and then $\eps\to0$, analogously as it
has been done in Section \ref{L2s}, we conclude from
\eqref{061805-26} that
\begin{equation}
  \label{071805-26}
\int_{\rH}\big({\cal P}_t^{V}1\big)^pd\nu\le {\frak C}_{*,1} 
\exp\left\{-\lambda(p)t\right\}+{\frak C}_{*,2} 
,\quad t\ge0.
\end{equation}
From the previous discussion, we already know that $\lim_{t\to+\infty}\|{\cal P}_t^{V}1-\bar
F_V\|_{L^1(\nu)}=0$. Using the Fatou lemma for some subsequence
$t_n\to+\infty$ we conclude therefore that $\bar F_V\in
L^p(\nu)$. This ends the proof of Theorem \ref{cor022404-26}.\qed

\appendix

\section{Functional inequalities for a multi-dimensional O-U process}

\subsection{Log Sobolev inequality}

\label{AppA}
   We assume
  that $N\times N$ matrix $B$  is symmetric and {non-negative} definite.
Suppose futhermore that    $\Sigma$ is $N\times N$ matrix 
 such that
 \begin{equation}
   \label{B-S}
B\Sigma=\Sigma B\quad\mbox{and}\quad B\Sigma^T=\Sigma^T B.
\end{equation}
To abbreviate we denote $A= \Sigma\Sigma^T$ and assume that there exists some unit vector $\phi_0\in \R^N$ such that    \begin{equation}
    \label{ker}
    {\rm ker}\,B={\rm ker}\, A={\rm span}\,({\rm
      \phi}_0). 
  \end{equation}
  Let
 $\nu_{B,A}(dx)=\Phi_{B,A} (x)\delta_0\big(\la x,\phi_0\ra\big)
dx$, where 
\begin{equation}
 \label{FBA0}
 \begin{split}
 &\Phi_{B,A}=\frac{1}{Z_{B,A}}\exp\left\{-A^{-1}Bx\cdot x\right\},\\
 &
 Z_{B,A}= \int_{\bbR^N}\exp\left\{-A^{-1}Bx\cdot x\right\}\delta_0\big(\la x,\phi_0\ra\big)dx.
 \end{split}
 \end{equation}{Here $A^{-1}Bx$ is well-defined for all $x\in\R^N$. To see why this is true, we  note that $AB=BA$ so the two matrices can be diagonalized simultaneously then it is enough to conclude using  \eqref{ker}.}
 Define the Dirichlet form  
\begin{equation}
  \label{DFA0}
\begin{split}
&{\cal
  E}_{B,A}
(f)= \frac12\int_{{\cal X}_N}ADf(x)\cdot Df(x) \nu_{B,A}(dx)
\end{split}
\end{equation}
for all $f \in{\cal D}\big( {\cal
  E}_{B,A}\big)$ - the set made of such $f$ for which   $Df$ exists a.e. and the integral is
finite.   In
all other cases we let ${\cal
  E}_{B,A}
(f)=+\infty$. Here the space
 \begin{equation}
   \label{XN}
   {\cal X}_N:=\big[x=(x_1,\ldots,x_N)\in\bbR^N:\,\la x,\phi_0\ra=0\big].
 \end{equation}
 
The following result follows from \cite[Section 5.5.1]{bakry}.
\begin{lemma}
  \label{lm013103-26aa}
    We have
    \begin{equation}
      \label{043103-26}
              {\cal
         E}_{B,A} (f)\ge\lam_B
         \int_{{\cal X}_N}
   f^2(   x )\log\Big(\frac{   f^2(    x)}{ \int_{{\cal X}_N}
    f^2(  y)
      \nu_{B,A}(dy)}\Big)
   \nu_{B,A}(dx) 
   \end{equation}
   for all $f\in L^2(\nu_{B,A})$, where
 \begin{equation}
   \label{lamB}
       \lam_B:=   \inf_{|\xi|=1,\,\la\xi,\phi_0\ra=0}\la
       B\xi,\xi\ra.
       \end{equation}
\end{lemma}

Consider the process $X_t=(X_t^{(1)},\ldots,X_t^{(N)})$ on $\bbR^N$ given by
\begin{equation}
  \label{1do-um}
dX_t=-BX_tdt+\Sigma dW_t,
\end{equation}
with $W_t=(w_t^{(1)},\ldots,w^{(n)}_t)$ and $w_t^{(j)}$ are
i.i.d. standard,  one dimensional Brownian motion.
One can see that then
  \begin{equation}
    \label{X0}
   d\la X_t,\phi_0\ra\equiv 0,\quad t\ge0.
 \end{equation}
 In this case the process $(X_t)$ satisfying if $X_0\in {\cal X}_N$ then $X_t$ stays
 in $  {\cal X}_N$. 
We have
\begin{equation}
  \label{DFA0}
\begin{split}
&{\cal
  E}_{B,A}^{(0)}
(f)=-\frac{d}{dt}_{|t=0}\EE\Big[f(X_t)f(X_0)\Big] ,\quad f\in {\cal D}\big( {\cal
  E}_{B,A}\big).
\end{split}
\end{equation}
  Let
$
H_t^{(B, A)}:=H(f_t|\nu_{B,A}),
$ where $f_t$ is the density of the distribution of $X_t$, solving
\eqref{1do-um}, w.r.t. $\nu_{B,A}$. After a straightforward
calculation we obtain
\begin{align*}
  &\frac{dH_t^{(B, A)}}{dt}
    =-4{\cal E}_{B,A}(\sqrt{f_t}).
\end{align*}
Hence, by \eqref{043103-26} (remembering that
$\int_{\bbR^N}f_td\nu_{B,A}=1$), we conclude the following  inequality
\begin{equation}
  \label{ei0}
  \begin{split}
    &\frac{dH_t^{(B,A)}}{dt}\le-
    2\lam_BH_t^{(B,A)} . 
  \end{split}
\end{equation}

  \subsection{Spectral gap and related estimates}

  \label{AppA1}

  The following estimate is a direct consequence of  \cite[Proposition 4.1.1]{bakry}.
   \begin{equation}
     \label{021105-26}
     \begin{split}
\int_{{\cal X}_N}\Big(g(x)&-\int_{{\cal X}_N} g(x') \nu_{B,A} (dx')\Big)^2 \nu_{B,A} (dx)
\\
&
\le
\frac{ 1}{2\lambda_B }
\int_{{\cal X}_N} 
A Dg(x)\cdot Dg(x)\,
\nu_{B,A} (dx),\quad g\in L^2(\nu_{B,A} ).
\end{split}
\end{equation}

 \begin{proposition}
   \label{prop021704-26}
Assume that \eqref{ker}.
Then, for every ${\bf a}\in\mathbb R^N$ and every $f\in L^2(\nu_{B,A} )$ such that
${{\cal E}_{B,A}(f)}<\infty$,
\[
\left|\int_{{\cal X}_N} ({\bf a}\cdot x)\,f^2(x)\,\nu_{B,A} (dx)\right|
\le
\sqrt{2}\,\big|A_0^{1/2}B_0^{-1}P_0{\bf a}\big|\,
\|f\|_{L^2(\nu_{B,A} )}\,
\sqrt{{\cal E} _{B,A}(f)},
\]
where $P_0=I-\phi_0\otimes\phi_0$ is the orthogonal projection onto ${\cal X}_N$, and
$A_0,B_0$ denote the restrictions of $A,B$ to ${\cal X}_N$.
\end{proposition}

\begin{proof}
Since $x\in{\cal X}_N$, we have
\[
{\bf a}\cdot x = (P_0{\bf a})\cdot x.
\]
Set
\[
C_0:=A_0^{-1}B_0.
\]
Then $C_0$ is symmetric positive definite on ${\cal X}_N$, and the invariant
measure can be written as
\[
\nu_{B,A}(dx)
=
\frac{1}{Z_{B,A}} e^{-C_0 x\cdot x}\,dx_{{\cal X}_N}.
\]
Hence
\[
\nabla \Phi_{B,A}(x)=-2C_0 x\,\Phi_{B,A}(x),
\]
and therefore
\[
(P_0{\bf a}\cdot x)\,\Phi_{B,A}(x)
=
-\frac12 (C_0^{-1}P_0{\bf a})\cdot \nabla \Phi_{B,A}(x).
\]
Integrating by parts over ${\cal X}_N$, we obtain
\begin{align*}
\int_{{\cal X}_N} ({\bf a}\cdot x) f^2(x)\,\nu_{B,A}(dx)
&=
-\frac12 \int_{{\cal X}_N} f^2(x)\,(C_0^{-1}P_0{\bf a})\cdot \nabla \Phi_{B,A}(x)\,dx \\
&=
\frac12 \int_{{\cal X}_N} (C_0^{-1}P_0{\bf a})\cdot \nabla(f^2)(x)\,\nu_{B,A}(dx) \\
&=
\int_{{\cal X}_N} f(x)\,(C_0^{-1}P_0{\bf a}\cdot \nabla f(x))\,\nu_{B,A}(dx).
\end{align*}
Using $C_0^{-1}=B_0^{-1}A_0$ and applying Cauchy--Schwarz,
\begin{align*}
\left|\int_{{\cal X}_N} ({\bf a}\cdot x) f^2(x)\,\nu_{B,A}(dx)\right|
&\le
\int_{{\cal X}_N}
|f(x)|\,|(C_0^{-1}P_0{\bf a})\cdot \nabla f(x)|\,\nu_{B,A}(dx) \\
&\le
|A_0^{{1/2}}{B_0}^{-1}P_0{\bf a}|\,
\|f\|_{L^2(\nu_{B,A})}
\left(\int_{{\cal X}_N} A_0\nabla f(x)\cdot \nabla f(x)\,\nu_{B,A}(dx)\right)^{1/2}.
\end{align*}
Since
\[
{\cal E}_{B,A}(f)
=
\frac12\int_{{\cal X}_N} A_0\nabla f\cdot \nabla f\,d\nu_{B,A},
\]
we conclude
\[
\left|\int_{{\cal X}_N} ({\bf a}\cdot x) f^2(x)\,\nu_{B,A}(dx)\right|
\le
\sqrt{2}\,|A_0^{1/2}{B_0}^{-1}P_0{\bf a}|\,
\|f\|_{L^2(\nu_{B,A})}\,
\sqrt{{\cal E}_{B,A}(f)}, 
\]
which completes the proof.
\end{proof}
  
\section{Proof of Theorem \ref{thm010904-26}: solution to a regularized Burgers equation}
\label{appB}

The case $V=0$ has been considered in   \cite[Theorem 2.1]{FQ}.
Suppose that $Q_\eps$ is the law of $ u^{(\eps)}(\cdot)$ - the unperturbed
stationary solution described in  
 Theorem \ref{thm010904-26} -  on the space ${\cal
   C}:=C\big([0,+\infty);C(\bT)\big) $. Let 
 $\mathbf u(t;\om):=\om(t)$, $\om \in{\cal C}$ be the canonical
process  and let $\big({\cal M}_t\big)_{t\ge0}$ be its  natural filtration.
{Under the measure $Q_\eps$,}  the process $\big(\tilde{\cal
  W}^{(\eps)}(t)\big)_{t\ge0}$ determined by
\begin{align*}
\int_{\bT}\tilde{\cal
                 W}^{(\eps)}(t,x)dx=0&\quad\mbox{and}\quad \la {\nabla} \tilde {\cal W}^{(\eps)}(t;\om),\varphi\ra=\la \mathbf u(t)
    ,\varphi\ra\\
  &
      -\int_0^t {\tfrac12}\Big[\la \mathbf u(s)  ,\Delta
    \varphi\ra -\la \mathbf u^2(s)\star
    \eta_2^{(\eps)} ,\nabla
    \varphi\ra \Big]ds,\quad \varphi\in C^\infty(\bT),
\end{align*}
is the Wiener process satisfying  \eqref{weps} (see \eqref{030904-26}).

 Define a Borel measure $Q_\eps^{(V)}$
on ${\cal C}$ such that for any bounded, ${\cal M}_T$-measurable
function $\Phi$  
\begin{equation}
  \label{020904-26}
  \int_{{\cal C}}\Phi(\omega) \,Q_\eps^{(V)}(d\omega)= \int_{{\cal
      C}}\Phi(\om) \exp\Big\{  \la {\tilde {\cal W}}^{(\eps)}(T;\om), V\ra-\tfrac12T\la R_\eps V,V\ra 
 \Big\} Q_\eps(d\om).
\end{equation}
Using It\^o formula one can verify that under this new measure $Q_\eps^{(V)}$, the process
$$
 {\cal W}^{(\eps)}(t;\om):= {\tilde {\cal W}}^{(\eps)}(t;\om)
-t   {V^{(\eta_2^{\eps})}},\quad t\ge0
$$
is $\big({\cal M}_t\big)$-adapted  Wiener process
with the covariance given by \eqref{weps}. Therefore, 
\begin{align}
  \label{011505-26}
&  \la {\nabla} {\cal W}^{(\eps)}(t;\om),\varphi\ra=\la \mathbf u(t)
    ,\varphi\ra -\int_0^t \Big[\tfrac12\la \mathbf u(s)  ,\Delta
    \varphi\ra \\
  &
  +\la \nabla V^{(\eta_2^{\eps})},
    \varphi\ra    -\tfrac12\la \mathbf u^2(s)\star
    \eta_2^{(\eps)} ,\nabla
    \varphi\ra \Big]ds,\quad \varphi\in C^\infty(\bT)\notag
\end{align}
The canonical process satisfies \eqref{eq:SBEV} in the same sense
as discussed in the proof of   Theorem  \ref{thm010904-26} in the case
$V=0$. 
The remaining part of the argument follows the same lines as in
the proof   of   Theorem  \ref{thm010904-26}.

From \eqref{011505-26} we conclude that ${\bf u}(t)$ is a mild solution of the equation for $Q_\eps^{(V)}$ a.s. $\om$, i.e.
\begin{equation}
  \label{SBEE3}
  \begin{split}
&   {\bf u}(t) = e^{1/2\Delta t}{\bf u}(0)   +  \int_0^t 
e^{1/2\Delta (t-s)} \big({\tfrac12}{\bf u}^2(s)\star
\nabla\eta_2^{(\eps)}+\nabla V^{(\eta_2^{\eps})}\big)  ds \\
&
+\int_0^t \nabla
e^{1/2\Delta (t-s)}d{ {\cal W}}^{(\eps)} (s) .
\end{split}
\end{equation}
For $M>0$, consider the Lipschitz function $\Phi_M: C(\bT)\to C(\bT)$
such that
$$
\Phi_M(u)=\begin{cases}
  u, & \|u\|_{L^2(\bT)}\le M,\\
  & \\
  0,& \|u\|_{L^2(\bT)}> M+1.
  \end{cases}
  $$
  Let ${\bf u}_M(t;\om)$ be the solution of 
\[
  \begin{split}
&   {\bf u}_M(t) = e^{1/2\Delta t}{\bf u}(0)   +  \int_0^t
e^{1/2\Delta (t-s)} \big({\tfrac12}[{\bf u}_M(s) \Phi_M\big({\bf
  u}_M(s)\big)]  \star
 \nabla\eta_2^{(\eps)}+\nabla V^{(\eta_2^{\eps})}\big)   ds \\
&
+\int_0^t \nabla
e^{1/2\Delta (t-s)}d{ {\cal W}}^{(\eps)} (s) .
\end{split}
\]
Thanks to the fact that ${\frak B}_M: C(\bT)\to C(\bT)$ given by
$$
{\frak B}_M[u]:=\big( u\Phi_M(u)\big)\star
 \nabla\eta_2^{(\eps)}
 $$
 is Lipschitz, for any ${\bf u}(0)$ and   a.s. realization of $
{\cal{ W}}^{(\eps)}$, there exists a unique solution of the above equation. It is
 the limit, uniform on compact time intervals of the Picard
 iteration scheme
\[
  \begin{split}
&   {\bf u}_{M,n+1}(t) = e^{1/2\Delta t}{\bf u}(0)   +  \int_0^t
e^{1/2\Delta (t-s)} \big(\mm{\tfrac12}{\bf u}_{M,n}(s) \Phi_M\big({\bf
  u}_{M,n}(s)\big)  \star
 \nabla\eta_2^{(\eps)}+\nabla V^{(\eta_2^{\eps})}\big)  ds \\
&
+\int_0^t \nabla
e^{1/2\Delta (t-s)}d\mm{\cal{ W}}^{(\eps)} (s) ,\quad n=0,1,\ldots,\\
&
{\bf u}_{M,0}(t) \equiv{\bf u}(0),\quad t\ge0.
\end{split}
\]
Let
$$
\tau_M:=\inf[t\ge0:\, \|\mm{{\bf u}_M}(t)\|_{L^2(\bT)}\ge M].
$$
Since $t\mapsto {\bf u}(t)$ satisfies the same equation as ${\bf u}_M$ for
$t\in[0,\tau_M]$ we have, by Gronwall's inequality that
$$
 {\bf u}_{M }(t) ={\bf u}(t),\quad t\in[0,\tau_M].
 $$
 Hence, for $Q_\eps^{(V)}$ we have
 \begin{equation}
   \label{uuM}
   \lim_{M\to+\infty}{\bf u}_{M }(t) ={\bf u}(t),\quad t\ge
   0\quad\mbox{uniformly on compact intervals.}
 \end{equation}
 {Fix any $m$, we have \[
 \lim_{M\to+\infty}Q_\eps \big[\tau_{M}\leq m\big]=\lim_{M\to+\infty} Q_\eps\bigg[\sup_{t\in[0,m]}\|{\bf u}(t)\|_{L^2(\bT)}\geq M\bigg]=0.
 \]
 By \eqref{020904-26}, 
 we have the same result holds for $Q_\eps^{(V)}$:
 \begin{equation}
   \label{uuM1}\lim_{M\to+\infty}Q_\eps^{(V)}\big[\tau_{M}\leq m\big]=0.
   \end{equation}
   In other words, with $Q_\eps^{(V)}-$probability $1$, there is no blow up in  $[0,m]$.}

   Now consider the equation {driven by the given noise $W^{(\eps)}$}
\[
  \begin{split}
&    u_{\eps,M}(t;u) = e^{1/2\Delta t}u   + \int_0^t
e^{1/2\Delta (t-s)} \big(\mm{\tfrac12}[u_{\eps,M}(s;u)  \Phi_M\big(u_{\eps,M}(s;u) \big)]  \star
 \nabla\eta_2^{(\eps)}+\nabla V^{(\eta_2^{\eps})}\big)   ds \\
&
+\int_0^t \nabla
e^{1/2\Delta (t-s)}d  W^{(\eps)} (s), 
\end{split}
\]
over the space $C(\bT)\times \Om$ with the product  probability
measure $\PP_{\nu^{(\eps)}}:=\nu^{(\eps)}\otimes \PP$. It has a unique solution for each $M$.
In fact the laws of $u_{\eps,M}(\cdot;u)$ and ${\bf u}_{M }(\cdot)$
over ${\cal C}$ coincide. 
Let
$$
\tau_{\eps,M}:=\inf[t\ge0:\, \|u_{\eps,M}(\cdot;u)\|_{L^2(\bT)}\ge M].
$$
We have
\[
  \begin{split}
&u_{\eps,M}(t;u)=u_{\eps,M'}(t;u),\quad
t\in[0,\tau_{\eps,M}]\quad \mbox{and}\\
&
\tau_{\eps,M}\le \tau_{\eps,M'}\quad \,0<M<M'.
\end{split}
\]
By the fact that they have the same law, we conclude
$$
\PP_{\nu_\eps}\big[\tau_{\eps,M}\le m\big]=Q_\eps^{(V)}\big[\tau_{M}\le
m\big],\quad m,M>0.
$$
Hence  
$$
  {\lim_{M\to+\infty}\PP_{\nu_\eps}\big[\tau_{\eps,M }\leq m\big]=\lim_{M\to+\infty}Q_\eps^{(V)}\big[\tau_{M}\leq m\big]=0,}
  $$
  by virtue of \eqref{uuM1}. {Since $m$ is arbitrary,} this proves the conclusion of the theorem.

\subsection*{Acknowledgement. } YG would like to thank Jonathan Mattingly for introducing the relevant problem to him. The research of YG is partially supported by the National
Science Foundation under grant no. DMS-2203014. T.K acknowledges the support of the NCN grant
2024/53/B/ST1/00286.

\bibliographystyle{amsalpha}

  \end{document}